%% file: main.tex
\documentclass[a4paper]{article}
\usepackage[utf8]{inputenc}
\usepackage{amsmath,amsthm,amssymb,amsfonts}
\usepackage[english]{babel}
\usepackage{pdfpages}
\usepackage{adjustbox}
\usepackage{tikz}
\usepackage{xcolor}

\theoremstyle{plain}
\newtheorem{theorem}{Theorem}
\newtheorem{lemma}[theorem]{Lemma}
\newtheorem{claim}[theorem]{Claim}
\newtheorem{defin}[theorem]{Definition}

\newtheorem*{remark}{Remark}

\usepackage{hyperref}
\usepackage{refcount}

\newcommand{\refoot}[1]{\kern0.12em\hyperref[#1]{\textsuperscript{\normalfont\footnotesize\ref*{#1}}}}
\newcounter{reusefn}

\title {
    Spanning triangulations in random graphs
}
\author {
    S. Vakhrushev\footnote{Department of Mathematics, Uppsala University, Uppsala, Sweden, and Moscow Institute of Physics and Technology, Dolgoprudny, Russia \hfill\break Corresponding author, email: \texttt{stepan.vakhrushev@math.uu.se}}, M. Zhukovskii\footnote{School of Computer Science, The University of Sheffield, Sheffield S1 4DP, United Kingdom}
}

\date{}

\usepackage[]{hyperref}

\definecolor{schemecolor}{HTML}{A58F9F}

\usepackage{subfigure}

\newcounter{image}
\newenvironment{image}[1][]{\refstepcounter{image} 
	\small Fig.~\theimage. #1 \rmfamily}{}

\usepackage{pgfplots}

\usepackage{hyphenat}
\begin{document}

\maketitle

\begin{abstract}
In 1991 Bollob\'{a}s and Frieze found the threshold for the emergence of a spanning triangulation of a triangle in the binomial random graph, up to a logarithmic factor. In this paper, we find the threshold probability for the emergence of a spanning triangulation of a $k$-gon for any $3\leq k\leq n$, up to a constant factor. 
\end{abstract}

\vspace{\baselineskip}

\input{introduction}

\input{tech_part}

\input{lower_bound}
\input{upper_bound}
\input{the_whole_border}
\input{conclusion}
\input{references}
\input{appendix}
\end{document}

%% file: introduction.tex
\section{Introduction}
 Threshold functions play a central role in the analysis of random graphs. Recall that in the {\it binomial random graph} $\mathbf{G}\sim G(n,p)$ edges between any pair of vertices from the $n$-vertex set $[n]:=\{1,\ldots,n\}$ appear independently with probability $p$. Formally, for a non-trivial monotone\footnote{A graph property is called increasing, if it is upward closed --- that is, addition of edges preserves the property. Downward closed properties are called decreasing. A property is monotone if it is either increasing or decreasing.} property $\mathcal{A}$, the {\it threshold} $p_c(\mathcal{A})$ is the unique $p$ for which $\mathbb{P}(\mathbf{G}\in \mathcal{A})= 1/2$. Significant progress has been made in determining thresholds for the presence of copies of a specific graph $F$ (see~\cite[Ch. 4, 7, 8]{bollobas}, \cite[Ch. 5-6]{freeze}, \cite[Ch. 3-4]{janson}). However in the case of arbitrary \textit{spanning} graphs $F$, the known upper and lower bounds on the order of magnitude of the threshold do not match. The exact asymptotics of the threshold is known in some specific cases --- e.g., for Hamilton cycles~\cite{ham_cycles2, ham_cycles}, spanning trees~\cite{sp_trees}, factors~\cite{factors} and, in particular, perfect matchings~\cite{matchs}; see also the survey~\cite{survey}. 
 
 As follows from the theorem of Park and Pham~\cite{park}, for any sequence of graphs $F=F(n)$ on $[n]$, the probability threshold $p_c(F):=p_c(\mathcal{A}_F)$ for the property $\mathcal{A}_F$ to contain an isomorphic copy of $F$, differs from the so called expectation threshold by at most a logarithmic factor. Expectation threshold is, roughly speaking, the best lower bound that can be achieved using the union bound. Although the expectation threshold is hard to compute for an arbitrary property, there has been some recent progress~\cite{DKP} on the graph version of the Kahn--Kalai conjecture that states that the expectation threshold can be replaced with the best possible lower bound that follows from applying Markov's inequality to the number of subgraphs isomorphic to a subgraph of $F$. The result of Dubroff, Kahn, and Park~\cite{DKP} states that the upper bound is within the $(\log n)^3$-factor from this ``graph expectation threshold''. Even though these results provide a fairly tight interval of possible values for the probability threshold, for specific graphs $F$, finding the {\it order of magnitude} of $p_c(F)$ is a challenging task.

 When a family of graphs $\mathcal{G}$ that generates an isomorphism-closed monotone property is significantly small, so that the size of an isomorphism class of a single repsentative in $\mathcal{G}$ is comparable with the size of the entire family, it is reasonable to expect that the probability threshold for the entire family is close to the probability threshold for such a representative. The class of properties considered in this paper is of such kind. As a warm up, consider the property of containing a square of a Hamilton path $F$. Asymptotics of $p_c(F)$ is known~\cite{zhuk} and it coincides with asymptotics of the expectation threshold. Note that $F$ can be embedded on the plane so that all its internal faces are triangles, and the external face has length $n$, see Fig.~\ref{comb_triang}. Let $\mathcal{G}$ be the set of triangulations of an $n$-cycle on $[n]$. Since the total number of squares of Hamilton paths is $\frac{n!}{2}$, which is ``close'' to $|\mathcal{G}|=n!\exp(O(n))$, the expectation threshold for $\mathcal{G}$ is only constant-factor away from $p_c(F)$. This immediately implies the same order of magnitude for the probability threshold for the entire $\mathcal{G}$.

\renewcommand{\figurename}{Fig.}
\begin{figure}[htbp]
    \centering
    \begin{minipage}{0.8\textwidth}
        \centering
        \includegraphics[width=\linewidth]{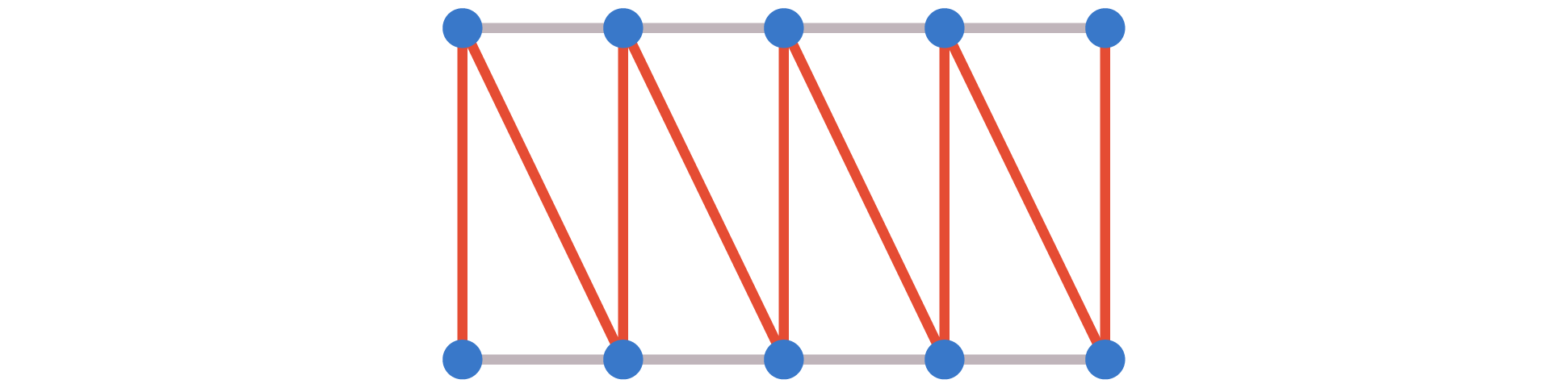} 
    \end{minipage}
    \caption{The second power of Hamilton path, the path is in red.}
    \label{comb_triang}
\end{figure}

 The purpose of this paper is to determine the order of magnitude of the probability threshold for the presence of a spanning triangulation of a $k$-gon, for any $k=k(n)\in\{3,\ldots,n\}$.
  We denote by $\mathcal{T}_{n,k}$ the set of all traingulations of a $k$-cycle on the set of vertices $[n]$, and we call each element $T\in\mathcal{T}_{n,k}$ an \textit{$\mathit (n, k)$-triangulation}. For each $T\in T_{n, k}$, we also denote by $S(T)$ the set of {\it internal} vertices of $T$, and by 
  \begin{equation}
   s=s(n):=n-k
   \label{eq:s-def}
  \end{equation} 
  --- the cardinality of this set. Note that Euler's formula for planar graphs implies that each $(n,k)$-triangulation contains exactly 
  $$
  m=m(n, k) :=3n-3-k
  $$ 
  edges. Now we are ready to formulate the main result of this paper:

\begin{theorem}
\label{th:main}
    Let $k=k(n)$ be a sequence such that $3\leq k \leq n$. Set 
    \begin{equation}
     \alpha = \alpha(n) := \frac{k}{n}.
     \label{eq:alpha-def}
  \end{equation} 
    Then the threshold probability for the appearance of at least one $(n,k)$-triangulation in $\mathbf{G}\sim G(n, p)$ equals $$p_c(\mathcal{T}_{n, k}) =\Theta\left(n^{-\frac{1}{3-\alpha}}\right)=\Theta\left(n^{-\frac{1}{2+s /n}}\right).$$
\end{theorem}
Therefore, the threshold is smoothly changing from $n^{-\frac{1}{2}}$ to $n^{-\frac{1}{3}}$ while $k$ decreases from $n$ to $3$.

\paragraph{Related work.} Bollob\'{a}s and Frieze~\cite{triangle-case} found $p_c(\mathcal{T}_{n, k})$ in the case of $k=3$ up to a $\Theta((\log n)^{1/3})$-factor: $\left(\frac{27e}{256}\right)^{\frac{1}{3}}n^{-\frac{1}{3}}\leq p_c(\mathcal{T}_{n,3})\leq 100(\log n)^{\frac{1}{3}}n^{-\frac{1}{3}}$.
Note that $\mathcal{T}_{n,3}$ consists of {\it maximal} planar graphs, i.e., planar graphs with $3n-6$ edges and $2n-4$ triangular faces, including the external face. In order to obtain the upper bound on $p_c$, the authors first recursively construct a certain very ``regular'' triangulation of a triangle and, second, partition the edges of the random graph in a corresponding recursive manner, at each step covering respective parts of a triangulation.

In~\cite{riordan}, Riordan established a general result that allows to achieve upper bounds --- which are fairly sharp for many properties --- on $p_c(F)$ for arbitrary graphs $F$. For example, it implies tight results for certain lattices, for hypercubes, and for the $k$-th power of Hamiltonian cycles for $k \geq 3$. Recently~\cite{kth_powers}, a refined version of Riordan's result was used to obtain {\it sharp threshold} for $k$-th powers of Hamilton cycles for $k\geq 4$. Riordan's result states that, under certain restrictions, $p_c=O(n^{-1/\gamma_F})$, where 
    $$
    \gamma_F= \max\limits_{F'\subset F: V(F')\geq 3} \frac{E(F')}{V(F')-2}.
    $$
    Let us note that $\gamma_F=3$ for any $3 \leq k \leq n$ and $F\in \mathcal{T}_{n, k}$. Indeed, one can just take any triangle as a subgraph $F'$ in the definition of $\gamma_F$, and on the other hand, each subgraph with $v\geq 3$ vertices contains at most $3v-6$ edges since it is planar. Thus, this result gives $p_c(\mathcal{T}_{n, k})=O(n^{-\frac{1}{3}})$ for all $3 \leq k \leq n$, implying the immediate improvement of the result of Bollob\'{a}s and Frieze: $p_c(\mathcal{T}_{n, 3})=\Theta(n^{-1/3})$, as in Theorem~\ref{th:main}. More generally, when $\alpha(n)=O(1/\log n)$ the lower bound (which we prove in Section~\ref{sc:lower}) equals $\Omega(n^{-\frac{1}{3}})$ as well, so Riordan's result provides the exact order of magnitude of $p_c$ in this case. 

Finally, as we have already observed, Theorem~\ref{th:main} for $k=n$ follows from known results. Indeed, Kahn, Narayanan and Park established~\cite{sec_power} the order of magnitude for the threshold for the square of a Hamilton cycle $F$, and quite recently the second author~\cite{zhuk} proved that $p_c(F)=(1+o(1))\sqrt{e/n}$. Since the square of a Hamilton path $F\in\mathcal{T}_{n,n}$ --- see Fig.~\ref{comb_triang} --- this result, together with our lower bound (see Section~\ref{sc:lower}), immediately implies $p_c(\mathcal{T}_{n,n})=\Theta(n^{-\frac{1}{2}})$, as in Theorem~\ref{th:main}.

Thus, our main contribution is the order of magnitude of $p_c(\mathcal{T}_{n,k})$ for all the intermediate values of $k$ (between $\omega(n/\log n)$ and $n-1$). We would also like to mention a related result on the threshold for the emergence of a maximal bipartite planar subgraph and maximal planar subgraph of a fixed girth by Fern\'{a}ndez, Sieger, and Tait~\cite{planar_girth}.

\vspace{\baselineskip}

\paragraph{Proof outline.} Recall~\eqref{eq:s-def},~\eqref{eq:alpha-def}. 
The lower bound in Theorem~\ref{th:main} is due to a  straightforward union bound argument, which is presented in Section~\ref{sc:lower}. For the upper bound, for every $k$, we give an explicit construction of a graph $T$ from $\mathcal{T}_{n,k}$ and show that, for $p\gg n^{-1/(2+s/n)}$, whp~\footnote{Henceforth {\it whp} stands for {\it with high probability}, i.e., with probability tending to $1$ as $n\to \infty$.} $\mathbf{G}\sim G(n,p)$ contains an isomorphic copy of $T$. We implicitly rely on the \textit{fragmentation} technique that has been used to prove the \textit{spread} lemma~\cite{spread_lm}, the Kahn-Kalai conjecture~\cite{kahn-kalai,park}, Talagrand’s fractional version of the conjecture~\cite{fract_expect_thresholds,talagran}, and to estimate the threshold probability for the square of a Hamilton cycle~\cite{sec_power,zhuk}. This method is applicable for monotone properties with good {\it spreadness} properties: a hypergraph $\mathcal{H}$ is called \textit{$q$-spread} if, for every subset $I$, $|\mathcal{H} \cap \langle I\rangle| \leq q^{|I|} |\mathcal{H}|$, where $ \langle I \rangle$ denotes all subsets of $V(\mathcal{H})$ that contain $I$ as a subset.
The proofs from~\cite{sec_power,zhuk} for the square of a Hamilton cycle were generalised in~\cite{f_cycles}, establishing tight (in the order of magnitude) upper bounds for $p_c(F)$ for graphs $F$ that generate well-spread properties in the following stronger sense. Let $F$ be any graph on $n$ vertices and let $\mathcal{F}=\mathcal{F}(F)$ be the set of all copies of $F$ in the complete graph $K_n$. For $\ell:=|E(F)|$ we view $\mathcal{F}$ as an $\ell$-uniform\footnote{Following the standard notation we call hypergraph {\it $\ell$-bounded} if cardinalities of its edges do not exceed $\ell$, and {\it $\ell$-uniform} if all its edges contain exactly $\ell$ vertices.\label{fn:hypergraph_not}} hypergraph on the vertex set $\Sigma := {[n] \choose 2}$, by identifying graphs with their sets of edges.
\setcounter{reusefn}{\value{footnote}}
\begin{defin}
    Let $q, \beta, \delta \in (0, 1)$. An $\ell$-bounded\refoot{fn:hypergraph_not} hypergraph $\mathcal{F}$ is $(q, \beta, \delta)$-superspread if it is $q$-spread and, for any subset $I\subset\Sigma$ of size $|I| \leq \delta \ell$, we have
    \begin{equation}
    \label{superspread}
        |\mathcal{F} \cap \langle I \rangle| \leq q^{|I|} {\ell}^{-\beta c_I} |\mathcal{F}|,
    \end{equation}
    where $c_I$ is the number of components of $I$ as a subgraph of $K_n$.
\end{defin}
\noindent
The main result of~\cite{f_cycles} establishes an upper bound for $p_c(F)$, relying only on the condition of superspreadness with suitable parameters.
\begin{theorem}[\cite{f_cycles}, Theorem 2.2]
\label{main_th}
Let $d, \beta, \delta, \varepsilon > 0$ with $\beta, \delta < 1$. There is a fixed constant $C_0=C_0(\varepsilon)$ such that, for all $C \geq C_0$ and $n\in \mathbb{N}$, the following holds. If $F$ is a subgraph of $K_n$ with maximum degree $\Delta(F) \leq d$,  ${\ell} := |E(F)| = \omega(1)$, and the hypergraph $\mathcal{F}=\mathcal{F}(F)$ is $(q, \beta, \delta)$-superspread with $q \geq 4{\ell} / (Cn^2)$, then, for $\mathbf{G} \sim G(n,Cq)$,
$$
\mathbb{P} (\text{there exists }F'\in\mathcal{F}\text{ such that }F' \subset \mathbf{G}) \geq 1 - \varepsilon .$$
\end{theorem}

If a triangulation $T=T(n,k)\in\mathcal{T}_{n,k}$ has a bounded maximum degree and is $(Kn^{-\frac{1}{3-\alpha}}, \beta, \delta)$-superspread with some constants $K>0$ and $0<\beta, \delta < 1$, then Theorem~\ref{main_th} immediately gives the desired upper bound for $p_c(T)$. This strategy works well when $s=\Omega(n)$ or $s\leq n^{\varepsilon}$ for some $\varepsilon < 1$. In order to establish this (super)spreadness property for such $s$, we present a new sufficient and relatively simple condition (Lemma~\ref{lm: q-spread property}) in terms of densities of subgraphs of $F$: roughly speaking, the densities of all subgraphs do not exceed the density of the entire graph $F$, with a small margin. This condition may be of independent interest. In the main case of $s=\Omega(n)$ (Section~\ref{main_case}), we construct such a triangulation $T$ with bounded degree that satisfies this condition. The triangulation consists of nested cycles of size $k$ and areas between any two consecutive cycles are ``evenly'' triangulated. Its symmetric structure allows to prove an isoperimetric inequality~\eqref{triang_ineq3} for its cycles, that allows to estimate the density of each subgraph of $T$ easily.

Unfortunately, in the case when $s(n)=o(1)$, it is not possible to directly apply this technology for any representative $T\in \mathcal{T}_{n, k}$, since each such $T$ either contains subgraphs with densities exceeding the density of entire $T$ (e.g. $K_4$), violating the requirements of Lemma~\ref{lm: q-spread property}, or has an internal vertex with unbounded degree, violating the requirements of Theorem~\ref{main_th}.  The first problem appears crucial when the number of internal vertices is large: $s=n^{1-o(1)}$. In this case, the factor $\ell^{-\beta c_I}$ in~\eqref{superspread} is not possible to get for $I$ consisting of small dense subgraphs (say, $K_4$). This factor is crucial, since the parameter $\beta$ is responsible for the number of fragmentation steps (constant $\beta$ implies bounded number of steps) and, apparently, the threshold probability for similar triangulations with larger number of internal vertices has an additional non-constant factor. To prove the superspreadness condition in the case $s=O(n^{\varepsilon})$ and apply Theorem~\ref{main_th}, we consider a triangulation $T$ with bounded degree whose `dense' patterns are $K_4$, and verify the superspread condition with $q=\Theta(n^{-1/2})=\Theta\left(n^{-\frac{1}{3-\alpha}}\right)$ directly by considering the contribution of disjoint unions of $K_4$ to the spread separately.

In the remaining intermediate regime $n^{1-o(1)} \leq s(n) \ll n$, where the order of the magnitude of the threshold can be higher than $n^{-1/2}$, we are forced to construct triangulations with unbounded maximum degree and therefore, we are not able to apply Theorem~\ref{main_th}. Instead, we follow another framework by Spiro~\cite{spiro}, who suggested a generalisation of the spreadness property that does not require $F$ to have a bounded degree: 
\begin{defin}
\label{def:spiro_spread}
    Let $0 < q \leq 1$ be a real number and ${\ell}_1 > \ldots > {\ell}_{t}$ be positive integers, where $t\geq 1$. A hypergraph $\mathcal{F}$ with $N$-vertex set $V$ is $(q; {\ell}_1, \ldots , {\ell}_{t})$-spread if $\mathcal{F}$ is non-empty, $\ell_1$-bounded, and if for all $A \subseteq V$ such that $A$ belongs to at least one edge in $\mathcal{F}$ and $\ell_{t'} \geq |A| \geq {\ell}_{t'+1}$ for some $1 \leq t' < t$, we have that, for all $i \geq {\ell}_{t'+1}$,
    \begin{equation}
      \label{eq:M_j-A} 
      M_i (A) := |{S \in \mathcal{F} : |A \cap S| \geq i}| \leq q^i |\mathcal{F}|.
    \end{equation}
\end{defin}
The following generalisation of Theorem~\ref{main_th} states that a random set of size $Ct q N$ contains an edge of an ${\ell}_1$-uniform $(q; \ell_1, \ldots , \ell_{t}, 1)$-spread hypergraph whp:

\begin{theorem}[\cite{spiro}]
\label{spiro-theorem}
    There exists an absolute constant $K_0$ such that the following holds. Let $\mathcal{F}$ be an $\ell_1$-uniform $(q; \ell_1, \ldots , \ell_{t}, 1)$-spread hypergraph. If $W$ is a set of $C t q N$ vertices chosen uniformly at random  with $C \geq K_0$, then
$$\mathbb{P}(W \text{ contains an edge of } \mathcal{F}) \geq 1 - \frac{K_0}{Ct} .$$
\end{theorem}
\noindent
In order to apply Theorem~\ref{spiro-theorem} to the hypergraph $\mathcal{F}(F)$, it is not essential that $F$ has bounded degree. In particular, Spiro's framework can also be applied in the first two regimes of large $s=\Theta(n)$ and small $s=n^{1-\Theta(1)}$. However, the formal verification of this generalised spreadness property is technically more difficult which makes this approach harder to apply despite its greater generality. Therefore, in order to simplify exposition, we consider the first two cases separately. We start from the more transparent approach that covers $s=\Theta(n)$ and $s=n^{1-\Theta(1)}$. Then, we verify the conditions of the generalised spreadness property in the last `intermediate' regime of $s=s(n)$ for a suitable triangulation $T \in \mathcal{T}_{n, k}$, which requires a more careful approximation of the number of ways to choose a subgraph with fixed parameters.

\paragraph{Organisation.}  Section~\ref{sc:tech} states two auxiliary claims. In Section~\ref{sc:lower} we prove the lower bound on the threshold probability. The next three sections are devoted to the proof of the upper bound, which is split according to the outline above: Section~\ref{main_case} considers $s(n)=\Omega(n)$, Section~\ref{small_s} deals with $s(n)=O(n^{\varepsilon})$ for some fixed $\varepsilon < 1$, and in Section~\ref{interm_case} we cover the remaining case $n^{1-o(1)} \leq s(n) = o(n)$. Section~\ref{sc:further} discusses further directions.

\paragraph{Notation.} For any graph $G$ we denote by $\Delta(G)$ the maximum degree of $G$, by $E(G)$ and $V(G)$ sets of edges and vertices of G correspondingly, and by $v(G)$ the total number of vertices $|V(G)|$ of $G$.

%% file: tech_part.tex
\section{Preliminaries}
\label{sc:tech}

In this section, we present several auxiliary propositions that we use in our proofs.

First, in our constructions of triangulations, we need to select edges of a cycle in a fairly even way. This is possible, for example, due to the following (suboptimal) proposition.
\begin{claim}
\label{cl:1}
    Let $0 \leq a < b$ be integers and let $C$ be a cycle of length $b$. There exists $U\subset E(C)$ of size $a$ such that for every arc $A\subset C$, the number of distinguished edges in it is $|E(A)\cap U|\leq |E(A)| \frac{a}{b}+1 .$
\end{claim}
Indeed, let $C=(c_1\ldots c_b)$. The choice of $U$ is obvious when $a=0$ or $b$ is divisible by $a$. Otherwise, one way to do it is to include in $U$ all vertices $c_{\lfloor ib/a\rfloor}$. 
\vspace{\baselineskip} 

 We will also rely on the following proposition that was used, for example, in the proof of Theorem~\ref{main_th} in~\cite[Lemma 2.3]{f_cycles} and in~\cite[Lemma 2.3]{sec_power}:

 \begin{lemma}
 \label{lm:subgraphs_count-rooted}
     For a graph $G$ of a maximum degree $\Delta$, the number of connected, $i$-edge subgraphs of $G$ containing a given vertex is less than $(e\Delta)^i$.
 \end{lemma}

%% file: lower_bound.tex
\section{Lower bound}
\label{sc:lower}

Let $N=N(n, k) := |\mathcal{T}_{n, k}|$ be the total number of $(n, k)$-triangulations. We roughly bound this value by the total number of labeled planar graphs with $m=m(n,k)$ edges. Turan's pioneering result~\cite{turan} states that unlabeled planar graphs with $n$ vertices and $m$ edges can be naturally encoded with $4m$ bits using depth-first search on a planar map. It implies that
$$N(n, k) \leq n! \cdot 2^{4m}.$$
Let us note that the exact number of rooted triangulations (i.e. with a fixed oriented edge on the boundary) of a $k$-gon is known --- see~\cite{brown},~\cite{triangle-number} and Tutte's classical result~\cite{tutte}. However, for our purposes of finding the order of magnitude of the threshold, a suboptimal bound above is sufficient.

Let $X$ be the number of $(n, k)$-triangulations in $\mathbf{G} \sim G(n, p)$. Recall~\eqref{eq:alpha-def}. Take, say, $p=\frac{1}{12}n^{-1/(3-\alpha)}$. By the union bound, we get:
$$
 \mathbb{P}[X > 0] \leq N(n, k) p^{m(n, k)} \leq en \left(\frac{n}{e}\right)^n (16p)^{3n-k-3}=n^{O(1)}e^{-n}(16/12)^{3n-k}\leq n^{O(1)}((4/3)^3/e)^n=o(1),
$$
so the lower bound in Theorem~\ref{th:main} is proved.

%% file: upper_bound.tex
\section{Upper bound:  \boldmath $s(n) = \Omega(n)$ }
\label{main_case}

The proof of the upper bound relies on the lemma below, which we believe may also be useful as an independent tool for determining thresholds of other graph properties.
\begin{lemma}
\label{lm: q-spread property}
    Let $\varepsilon, \delta > 0$. Let $F$ be a spanning subgraph of $K_n$ with $\Delta(F) \leq d = const$ and $C_1n \leq |E(F)| \leq C_2n$ for some constants $C_1, C_2 \geq 1$. Suppose that, for some $q>0$, and for every connected non-empty subgrpaph $I$  of $F$ the following inequalities hold:
    \begin{equation}
    \label{triang_eq}
    \begin{cases}
        v(I) \geq |I|/q + 1 + \varepsilon,\text{ if } |I| \leq \delta n, \\
        v(I) \geq |I|/q + 1,\text{ otherwise. }
    \end{cases}
    \end{equation}
    Then the hypergraph $\mathcal{H}$ of all copies of $F$ is $\left(8d C_2^{\varepsilon}e^{\varepsilon+1/q} n^{-1/q}, \varepsilon, \frac{\delta}{C_1}\right)$-superspread.
\end{lemma}
The proof of this lemma follows along similar lines as the proof in~\cite{f_cycles}. Therefore, we postpone it to Appendix~\ref{appendix}. We also stress that everywhere in the paper we apply this lemma with $q=3-\alpha$.

\begin{remark}
 \cite[Theorem 1.2]{zhuk} states the order of magnitude of $p_c$ for all $d$-regular graphs $F$ that satisfy a weaker version of~\eqref{triang_eq}, where the bound $|I|\leq\delta n$ is changed to $|I|=O(\log n)$. Nevertheless, the range of possible applications of Lemma~\ref{lm: q-spread property} is wider --- for instance, it can be applied for degenerate graphs.
\end{remark}

It then remains to derive the upper bound in Theorem~\ref{th:main} from Lemma~\ref{lm: q-spread property}. Assume $T=T(n,k)\in\mathcal{T}_{n,k}$ is chosen such that the $m$-uniform hypergraph $\mathcal{H}=\mathcal{H}(n, k)$ of all copies of $T$ satisfies the assumptions of Lemma~\ref{lm: q-spread property} with $q=3-\alpha$, and therefore it is superspread with parameters, as in the conclusion of the lemma. This implies $p_c(\mathcal{T}_{n,k})=O(n^{-1/q})$ by Theorem~\ref{main_th}, as needed.

Therefore, it remains to construct $T$ with a bounded maximum degree (we will show that it is possible with $d=7$) and satisfying~\eqref{triang_eq} with $q=3-\alpha$. Note that any $C\geq 3$ is sufficient for the remaining condition on the number of edges in the lemma. We split the proof into two cases --- $n\geq 2k$ and $n<2k$.

\subsection{\boldmath $n \geq 2k$}

\paragraph{Construction of $T(n, k)$, see Fig.~\ref{bign_examples}.} Let $n=ck+r$, where $c, r \in \mathbb{N}$, $c \geq 2$, and $0\leq r\leq k-1$. We start with a sequence of $c$ nested cycles $C_1,\ldots,C_{c-1},C^{\star}_c$ of size $k$ each, where the outer cycle $C_1$ corresponds to the boundary of the polygon. We enumerate their vertices as follows: let $C_i=(v_{i,1}v_{i,2}\ldots v_{i,k}v_{i,1})$, for $i\in[c-1]$, and let $C^{\star}_c=(v^{\star}_{c,1}v^{\star}_{c,2}\ldots v^{\star}_{c,k}v^{\star}_{c,1})$ be the last interior cycle.
Then we fill the space between each pair of two subsequent cycles $(C_i$, $C_{i+1})$, $i\in[c-2]$, and $(C_{c-1}, C_c^{\star})$, with an `even' triangulation by drawing edges $\{v_{i,j},v_{i+1,j}\},$ $\{v_{i,j},v_{i+1,j+1}\}$ for $i\in[c-2]$ and $\{v_{c-1,j},v^{\star}_{c,j}\},$ $\{v_{c-1,j},v^{\star}_{c,j+1}\}$, where $j\in[k]$ (the addition in indices is modulo $k$). We denote the obtained graph --- a prototype of the desired triangulation --- by $T^{\star}$. 

We have used $ck$ vertices so far, and now we insert $r$ final vertices on the interior $c$-th cycle; we call these vertices \textit{residual} and denote their set as $R$.  To do this, let us distinguish $r$ ``evenly distributed'' edges on $C^{\star}_c$ according to Claim~\ref{cl:1}. Subdivide each distinguished edge $\{v^{\star}_{c,j},v^{\star}_{c,j+1}\}$ by a residual vertex $u_j$, denote the new cycle obtained from $C^{\star}_c$ by $C_c$, and draw an edge from $u_j$ to $v_{c-1,j}$. Therefore, we have added $r$ vertices, and the space between $C_{c-1}$ and $C_c$ remains triangulated. For consistency, we now renumber the vertices of the extended $C_c=(v_{c,1}v_{c,2}\ldots v_{c,(k+r)}v_{c,1})$, respecting the order of $v^{\star}_{c,j}$.

Finally, we triangulate the internal cycle and get the final triangulation $T$. It is only essential to keep all the degrees bounded. For example, one could use a `comb' path $v_{c,1}v_{c,(k+r-1)}v_{c,2}v_{c,(k+r-2)}\ldots$, as in the square of a Hamilton cycle, see Fig.~\ref{comb_triang}. Depending on the parity of $k+r$, the last vertex of this path would be $v_{c,(k+r)/2+1}$ or $v_{c,(k+r-1)/2}$. We denote the set of edges of this triangulation (without the boundary cycle $C_c$) by $E_{int}(T)$.

\renewcommand{\figurename}{Fig.}
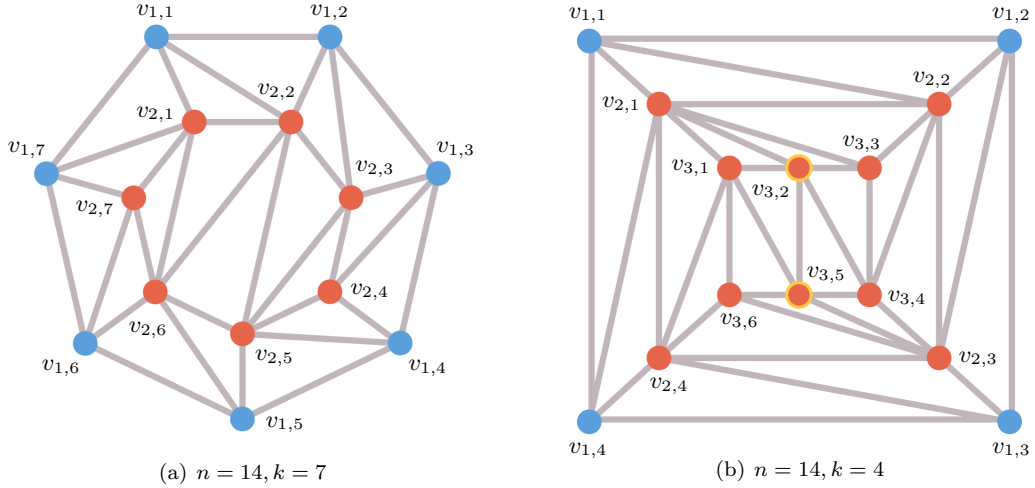
\begin{figure}[htbp]
\centering
	\subfigure[$n=14, k=7$]
	{
		\begin{minipage}{0.4\textwidth}
			\centering
			\begin{tikzpicture}
				\node[inner sep=0] at (0,0)
					{\includegraphics[width=\linewidth]{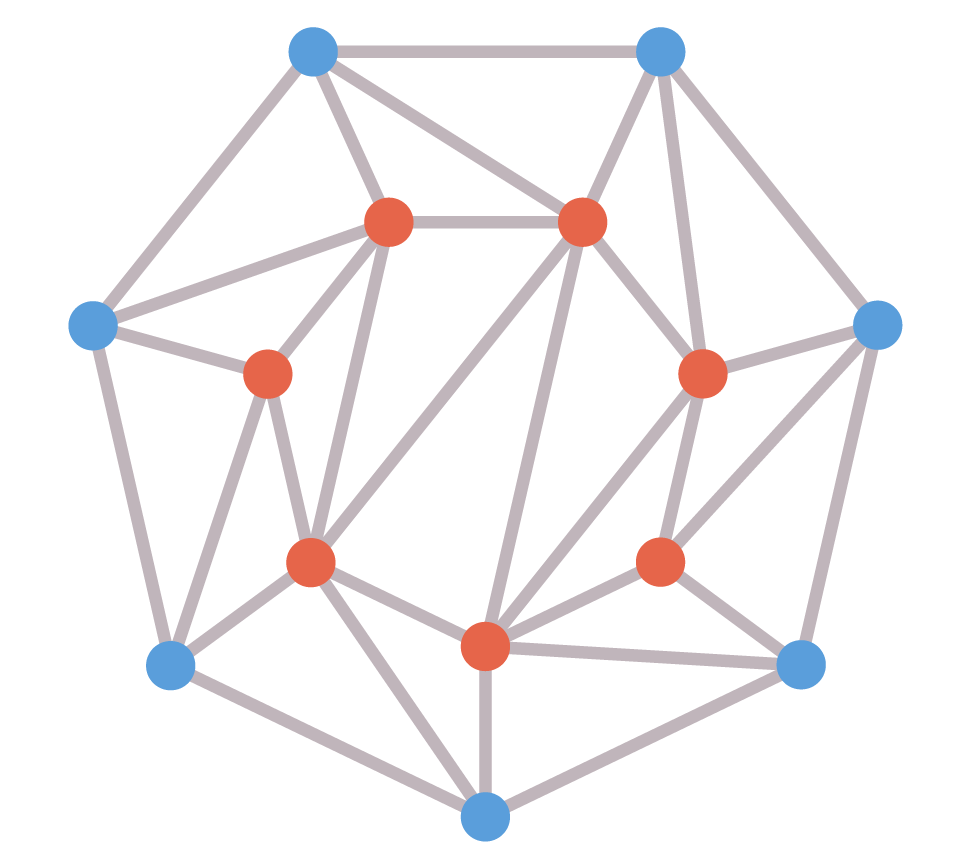}};
				\small
				\node[fill=none, inner sep=2pt] at (1.15,2.85) {$v_{1,2}$};
                \node[fill=none, inner sep=2pt] at (2.82,1.03) {$v_{1,3}$};
                \node[fill=none, inner sep=2pt] at (2.45,-1.85) {$v_{1,4}$};
                \node[fill=none, inner sep=2pt] at (0.56,-2.6) {$v_{1,5}$};
                \node[fill=none, inner sep=2pt] at (-2.4,-1.85) {$v_{1,6}$};
                \node[fill=none, inner sep=2pt] at (-2.85,1.03) {$v_{1,7}$};
                \node[fill=none, inner sep=2pt] at (-1.15,2.85) {$v_{1,1}$};

                \node[fill=none, inner sep=2pt] at (0.45,1.8) {$v_{2,2}$};
                \node[fill=none, inner sep=2pt] at (1.75,0.78) {$v_{2,3}$};
                \node[fill=none, inner sep=2pt] at (1.68,-0.85) {$v_{2,4}$};
                \node[fill=none, inner sep=2pt] at (0.42,-1.72) {$v_{2,5}$};
                \node[fill=none, inner sep=2pt] at (-1.25,-1.35) {$v_{2,6}$};
                \node[fill=none, inner sep=2pt] at (-1.95,0.25) {$v_{2,7}$};
                \node[fill=none, inner sep=2pt] at (-1.15,1.5) {$v_{2,1}$};
			\end{tikzpicture}
            \vspace{-0.2cm}
		\end{minipage}
	}
	\hspace{0.5cm}
	\subfigure[$n=14, k=4$]
	{
		\begin{minipage}{0.4\textwidth}
			\centering
			\begin{tikzpicture}
				\node[inner sep=0] at (0,0)
					{\includegraphics[width=\linewidth]{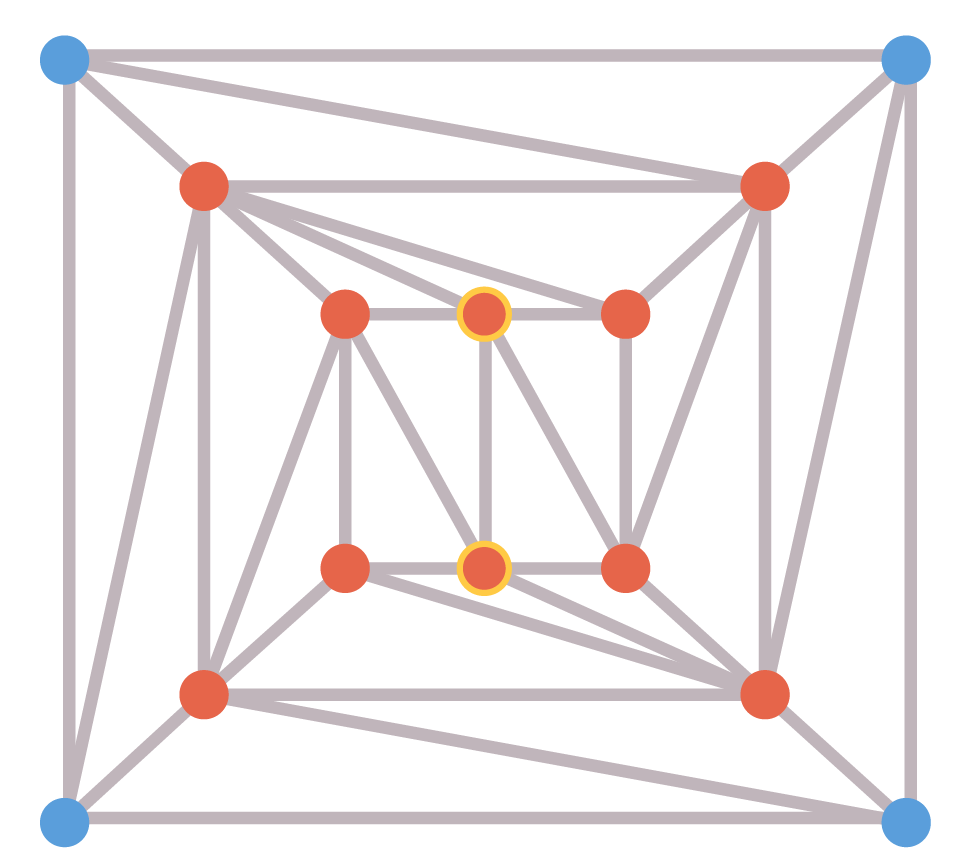}};
				\small
				\node[fill=none, inner sep=2pt] at (2.8,2.8) {$v_{1,2}$};
                \node[fill=none, inner sep=2pt] at (2.8,-2.95) {$v_{1,3}$};
                \node[fill=none, inner sep=2pt] at (-2.8,-2.95) {$v_{1,4}$};
                \node[fill=none, inner sep=2pt] at (-2.8,2.8) {$v_{1,1}$};
                \node[fill=none, inner sep=2pt] at (1.75,1.95) {$v_{2,2}$};
                \node[fill=none, inner sep=2pt] at (2.36,-1.75) {$v_{2,3}$};
                \node[fill=none, inner sep=2pt] at (-1.73,-2.1) {$v_{2,4}$};
                \node[fill=none, inner sep=2pt] at (-2.36,1.62) {$v_{2,1}$};
                \node[fill=none, inner sep=2pt] at (0.82,1.1) {$v_{3,3}$};
                \node[fill=none, inner sep=2pt] at (1.42,-0.95) {$v_{3,4}$};
                \node[fill=none, inner sep=2pt] at (-0.8,-1.25) {$v_{3,6}$};
                \node[fill=none, inner sep=2pt] at (-1.44,0.78) {$v_{3,1}$};
                \node[fill=none, inner sep=2pt] at (0.35,-0.63) {$v_{3,5}$};
                \node[fill=none, inner sep=2pt] at (-0.38,0.46) {$v_{3,2}$};
			\end{tikzpicture}
            \vspace{-0.5cm}
		\end{minipage}
	}
	\caption{Examples of triangulations for $n\geq 2k$: $k$-gone boundary is blue, internal vertices are red, and residual vertices are outlined in yellow.}
	\label{bign_examples}
\end{figure}

It is clear that the maximum degree of the constructed $(n, k)$-triangulation $T=T(n, k)$ is $\Delta(T) \leq 7$. The ``lattice structure'' of the constructed triangulation will be useful in the further proof.

\vspace{\baselineskip}

\paragraph{Checking condition~\eqref{triang_eq} for $T(n, k)$.} We first show that it suffices to prove the following isoperimetric inequality.

\begin{claim}
\label{cl:first_case}
    Suppose that for every simple cycle $C\subseteq T$ with length $\ell \leq k-1$ and $v$ vertices 'inside' $C$ (including vertices of $C$) the following inequality is true:
    \begin{equation}
    \label{triang_ineq3}
        \frac{\ell-1/3}{v} \geq \frac{k}{n} .
    \end{equation}
    Then $T$ satisfies the condition $(\ref{triang_eq})$ with $q=3-\frac{k}{n}$, $\varepsilon=\frac{1}{9}$, and $\delta=\frac{8}{9}$.
\end{claim}
    
\begin{proof}
    Note that, since $q>1$, it is enough to prove the condition (\ref{triang_eq}) for fragments $I\subseteq T$ without bridges. Indeed, if after removing a bridge between $I_1,I_2$ we are able to prove the inequality  separately for $I_1$ and $I_2$, then summing up the resulting inequalities we obtain the inequality for the initial subgraph with a bridge. Now let $C$ be the boundary of the outer face of a bridge-free connected subgraph $I$. We can also reduce it to the case when $C$ is a simple cycle. Indeed, consider the graph where vertices are subcycles of $C$ that are adjacent if the cycles have a common vertex. Since this graph is a tree, summing up the inequalities as in (\ref{triang_eq}), for every cycle, as above, we get the required inequality for $C$.

    Let $\ell$ be the length of $C$ and $v=v(I)$ be the total number of vertices in $I$. The number of edges in $I$ is at most the number of edges in a triangulation of the $\ell$-gon $C$ with the total number of vertices $v$: $|I| \leq 3v-3-\ell$. Therefore, the inequality
    \begin{equation}
    \label{eq:claim_iso_v-ell}
    v \geq \frac{3v-3-\ell}{3-\frac{k}{n}} + 1 + \frac{1}{9},
    \end{equation}
    implies the first inequality in (\ref{triang_eq}) with parameters $q,\varepsilon,\delta$ specified in the statement of the claim. The latter inequaility is equivalent to
    $$\frac{k}{n}v \leq \ell - \frac{1}{3} +\frac{10k}{9n}.$$
    Firstly, note that this inequality holds when $\ell \leq k-1$ due to the assumption (\ref{triang_ineq3}). Secondly, it is obvious that the inequality  (\ref{triang_ineq3}) holds for all simple cycles of length $\ell \geq k+1$ and for $\ell=k$ and $v\leq n-\frac{n}{3k}$. Therefore, we have proved (\ref{triang_eq}) except for the case when $v > n-\frac{n}{3k} \geq \frac{8}{9}n$ and $\ell=k$. Assume this last case. Since $I$ is connected and has a cycle, $|I|\geq v\geq \frac{8}{9}n$, as in the condition in the second case of~\eqref{triang_eq}. Since $\ell=k$, in the same way as in~\eqref{eq:claim_iso_v-ell}, it suffices to prove 
    $$v \geq \frac{3v-3-k}{3-\frac{k}{n}} + 1$$
    which is immediate since $v < n + 1$.
\end{proof}
   
Thus, it is enough to prove that there are no simple cycles that bound fragments with densities higher than the density of the entire triangulation. Let $C$ be a simple cycle in $T(n,k)$ of length $3 \leq \ell \leq k-1$ and with $v$ vertices `inside', including the vertices of $C$. Let us represent it as a (cyclic) sequence of consecutive arcs $C=A_1P_1A_2P_2\ldots A_tP_tA_1$, where 
 \begin{itemize}
 \item each $P_i$ is a non-empty subpath that lies entirely inside the (sub)triangulation bounded by $C_c$, i.e, $E(P_i) \subseteq E_{int}(T)$, and 
 \item $A_i$ are paths without edges in $E_{int}(T)$.
 \end{itemize}
 Each path $A_i$ begins and ends at a vertex of $C_c$. In the special case where the cycle $C$ does not intersect $E_{int}(T)$ we represent it as $A_1P_1A_1$, where $P_1 \subset C$ is some non-empty subpath that we specify below and $A_1 = C \backslash P_1$.

We denote by $M_i$ the planar subgraph of $T(n,k)$ which is separated by arc $A_i$ and internal cycle $C_c$ or the path $P_1$ in case when $C$ does not intersect internal triangulation ($M_i$ includes its boundary). Let $\ell_i$ be the length of arc $A_i$ and $v_i=v(M_i)$. Note that $v \leq \sum\limits_{i=1}^t v_i + \sum\limits_{i=1}^t (|P_i|-1)$ and $\ell \geq \sum\limits_{i=1}^t (\ell_i + |P_i|)$. We get
\begin{align*}
 \frac{\ell-1/3}{v}\geq
 \frac{\sum_{i=1}^t(\ell_i+|P_i|)-1/3}{\sum_{i=1}^t v_i + \sum_{i=1}^t (|P_i|-1)}&\geq
 \frac{\sum_{i=1}^t\ell_i+\sum_{i=1}^t (|P_i|-1)+2t/3}{\sum_{i=1}^t v_i + \sum_{i=1}^t (|P_i|-1)}\\
 &=
 \frac{\sum_{i=1}^t(\ell_i+2/3)+\sum_{i=1}^t (|P_i|-1)}{\sum_{i=1}^t v_i + \sum_{i=1}^t (|P_i|-1)}.
\end{align*}
Since $|P_i|\geq 1$, then, if $\frac{\sum_{i=1}^t(\ell_i+2/3)}{\sum_{i=1}^t v_i}\geq 1$, we get that
$$
\frac{\sum_{i=1}^t(\ell_i+2/3)+\sum_{i=1}^t (|P_i|-1)}{\sum_{i=1}^t v_i + \sum_{i=1}^t (|P_i|-1)}\geq 1>\frac{k}{n},
$$
as required. Otherwise,
$$
\frac{\sum_{i=1}^t(\ell_i+2/3)+\sum_{i=1}^t (|P_i|-1)}{\sum_{i=1}^t v_i + \sum_{i=1}^t (|P_i|-1)}\geq \frac{\sum_{i=1}^t(\ell_i+2/3)}{\sum_{i=1}^t v_i},
$$
and so it suffices to prove that, for all $i$,
\begin{equation}
\label{arc_ineq}
    \frac{\ell_i + \frac{2}{3}}{v_i} \geq \frac{k}{n}
\end{equation}
in order to get the required inequality (\ref{triang_ineq3}).

\vspace{\baselineskip}

Now let us introduce natural parameters $h$ and $w$ which represent `height' and `width' of the fragment $M_i$, respectively --- see Fig.~\ref{fig:new}. In order to define the `width' parameter, recall the definitions of $T^{\star}$ and $R$. We also denote $v_{ij}^{\star}=v_{ij}$ for any $i\in[k], j\in[c-1]$, for the sake of convenience of presentation. Let $Y$ be the set of `angle' coordinates for non-residual vertices in $M_i$, that is, 
$$Y = \{y: \exists z\in[c] \text{ such that } v^{\star}_{z,y} \in M_i \backslash R\}.$$
Let us observe that $Y$ forms a set of consecutive values modulo $k$, according to the definition of $T$. We set $w := |Y|-1$, that is the maximum difference in `angle' coordinates for vertices in $M_i$. Thus,
\begin{equation}
\label{eq:cycles_intersects}
    |C_j \cap M_i| \leq w+1 \text{ for any }j \in [c-1] \quad\quad \text{ and } \quad\quad |C_c^{\star} \cap M_i| \leq w+1.
\end{equation}
Note also that $w \leq k-2$ since $|Y| \leq \ell_i +1\leq \ell \leq  k-1$. This implies, in particular, that if some vertex $v \in M_i \cap C_j$ for some $j\in[c]$, then $V(A_i) \cap V(C_j) \neq \emptyset$. Indeed, otherwise $C_j$ lies entirely `inside' $C$ and then $|Y|=k$. Therefore, in what follows we only consider vertices that belong to cycles $C_j$, $j\in[c]$, which $A_i$ intersects and we denote by $h$ the number of such cycles. Note that, if $h=1$, the inequality (\ref{arc_ineq}) trivially holds since 
$$
 \frac{\ell_i+\frac{2}{3}}{v_i}=\frac{v_i-1+\frac{2}{3}}{v_i}\geq\frac{2}{3}>\frac{1}{2}\geq\frac{k}{n}.
$$
So we further assume $h \in [2, c]$. 

\vspace{\baselineskip}

Using these parameters, we are now able to bound the length $\ell_i$ of $A_i$ and the size $v_i$ of the fragment $M_i$. By our construction and Claim~\ref{cl:1}, we bound the number of residual vertices which lie in $M_i$ as
\begin{equation}
\label{eq: residual}
    |R \cap M_i| \leq (w+2) \frac{r}{k} + 1.
\end{equation}
In the last bound, we add 2 to $w$ since the two farthest intersections of $A_i$ with $C_c$ could be in residual vertices. Therefore, using~\eqref{eq:cycles_intersects}, \eqref{eq: residual}, we get
\begin{equation}
\label{total_bound}
    v_i = |M_i| = \sum_{j\in[c-1]}|C_j \cap M_i|+ |C_c^{\star} \cap M_i| + |R \cap M_i| \leq (w+1)h + (w+2) \frac{r}{k}+1.
\end{equation}

Now we  bound $\ell_i$. Let us consider two vertices $v, u \in M_i \backslash R$  with the largest difference $w$ in their `angular' coordinates. It is obvious that any pair of such $v, u$ lies on $A_i$.   Before moving on, we need to specify the choice of $P_1$ in the case when $C$ does not intersect $E_{int}(T)$. Let $M$ be the set of vertices lying `inside' the cycle $C$, and let us denote by $\mathcal{P}$ the set of all pairs $v,u \in M \backslash R$ whose `angular' coordinates differ by $w$. The only essential  property that we require is that the endpoints of $A_1$ (which are the endpoints of $P_1$ as well) have the same `radial' coordinate and that $A_1$ contains at least one pair of vertices from $\mathcal{P}$. For doing this, we choose the pair $(v, u) \in \mathcal{P}$ with the smallest difference in their `radial' coordinates. If $v$ and $u$ have the same `radial' coordinate, such a choice of $P_1$ is clear: we choose an arc between $u$ and $v$ on $C$ (so then the other arc is $A_1$). Otherwise, let $v=v^{\star}_{z, j}$ and suppose without loss of generality that $v$ has a larger `radial' coordinate than $u$. Note that due to the choice of $v$ and $u$ the vertex $v_{z-1, j}$ does not lie on $C$, so at least one of the two neighbors of $v$ on $C$ has its `radial' coordinate at least $z$. Walking along $C$ from $v$ towards this neighbour, we find a vertex $\tilde{v}$ that has the same `radial' coordinate $z$ as $v$. Let $P_1$ be the traversed arc of $C$ between $v$ and $\tilde{v}$ and let $A_1 = C \backslash P_1$. This choice satisfies our condition. Thus, whether $C$ intersects $E_{int}(T)$ or not, we have an arc $A_i$ such that its' endpoints have the same `radial' coordinate and it contain vertices $v, u$ with the `angle' difference $w$.

\renewcommand{\figurename}{Fig.}
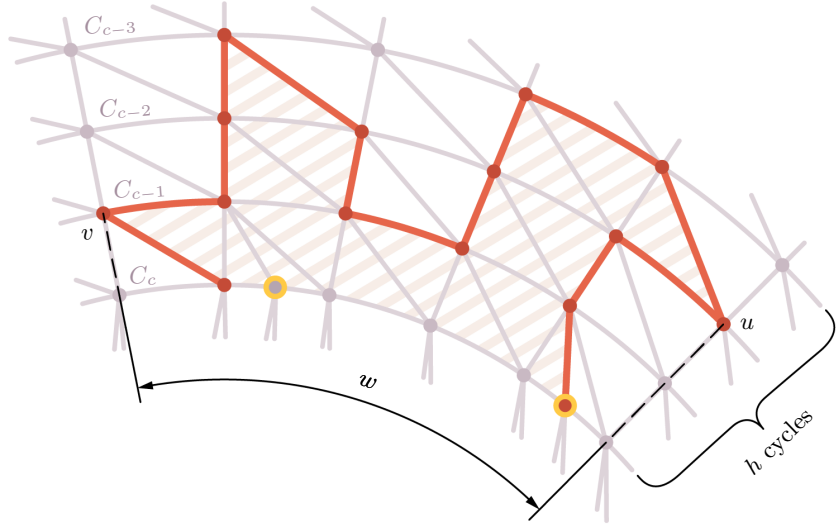
\begin{figure}[htbp]
    \centering
    \begin{tikzpicture}
        \node[inner sep=0] at (0,0)
            {\includegraphics[width=0.7\textwidth]{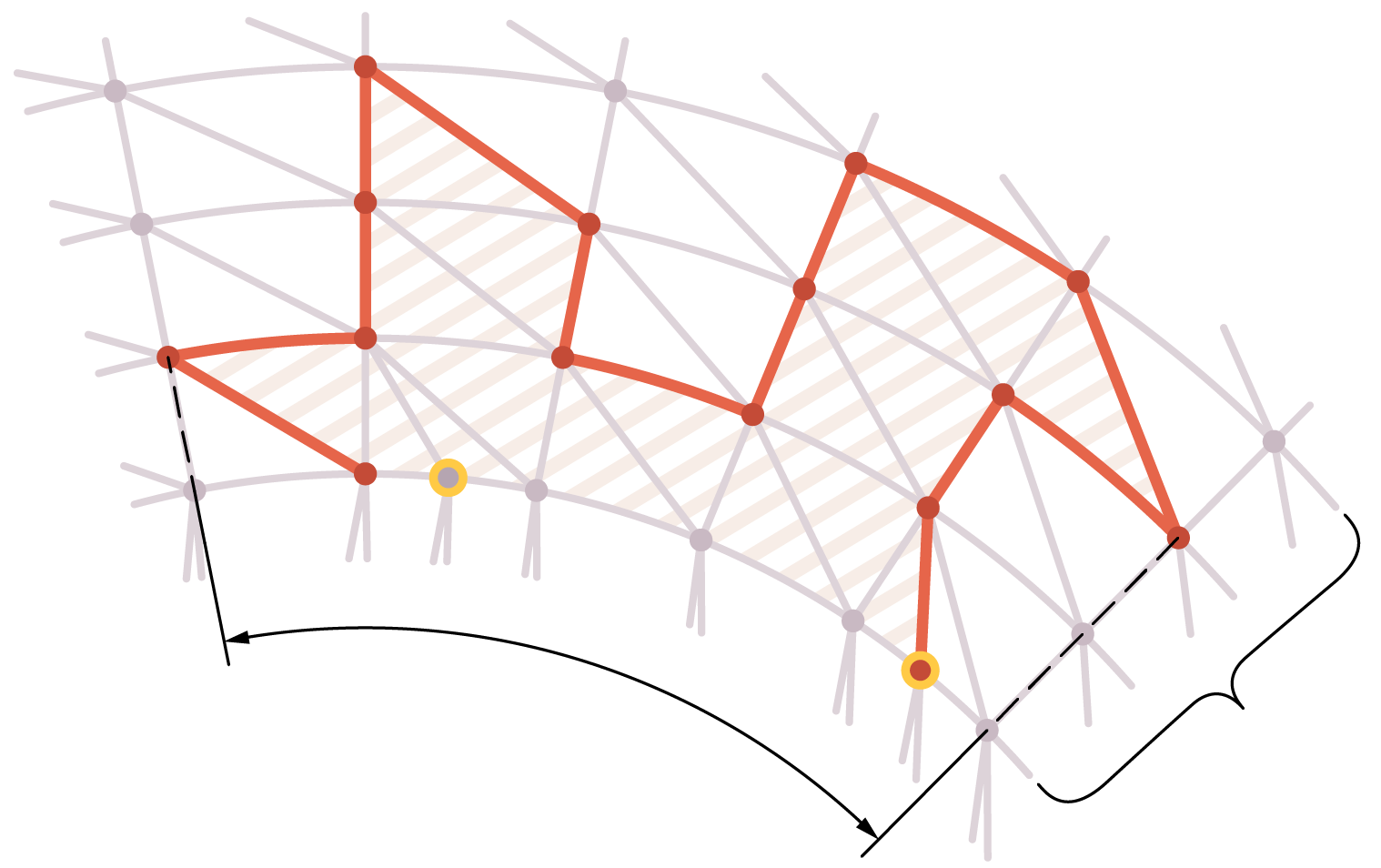}};
        \small
        \node[fill=none, inner sep=2pt, rotate=44] at (4.7,-2.4) {$h$ cycles};

        \node[fill=none, inner sep=2pt, rotate=-18] at (-0.75,-1.6) {$w$};

        \node[fill=none, inner sep=2pt, rotate=-18] at (-0.75,-1.6) {$w$};

        \node[fill=none, inner sep=2pt, rotate=5] at (-4.45,0.35) {$v$};

        \node[fill=none, inner sep=2pt, rotate=0] at (4.27,-0.84) {$u$};

        \node[fill=none, inner sep=2pt, rotate=5, text=schemecolor] at (-3.72,-0.2) {$C_c$};
        \node[fill=none, inner sep=2pt, rotate=5, text=schemecolor] at (-3.76,0.93) {$C_{c-1}$};
        \node[fill=none, inner sep=2pt, rotate=5, text=schemecolor] at (-3.95,2.0) {$C_{c-2}$};
        \node[fill=none, inner sep=2pt, rotate=5, text=schemecolor] at (-4.15,3.1) {$C_{c-3}$};
        
    \end{tikzpicture}
    \caption{Example of an arc $A_i$ (in red) and a subgraph $M_i$ (shaded): here $\ell_i=14, v_i=19$.}
    \label{fig:new}
\end{figure}

Getting back to $\ell_i$, we consider a subarc $B_i \subseteq A_i$ connecting $v$ and $u$. We find a set $D_i \subseteq E(B_i)$ of $w$ edges $\{v',u'\}$, where $v'$ is closer to $v$ on $D_i$ and the `angle' coordinate of $v'$ is less than the `angle' coordinate of $u'$, in the clockwise direction. There are $h-1$ pairs of consecutive cycles $C_j$ and $C_{j+1}$ that $A_i$ intersects. Since the endpoints of $A_i$ have the same `radial' coordinate, for each such pair $(C_j,C_{j+1})$, we are able to find a pair of closest (in $A_i$-distance) edges $e_1, e_2$ along the arc $A_i$ that connect vertices between $C_j$ and $C_{j+1}$. We note that at least one of the edges $e_1=\{v_1',u_1'\}, e_2=\{v_2',u_2'\}$ is not in $D_i$. Indeed, by the construction of $T$, the `angle' coordinate does not decrease when moving from $C_j$ to $C_{j+1}$. Therefore, assuming that the `angle' coordinates of $v'_1,v'_2$ are less than the `angle' coordinates of $u'_1,u'_2$ respectively and that $u'_1$ and $v'_2$ belong to $C_{j}$, we get that $v'_1$ cannot belong to $C_{j+1}$. So $A_i$ includes all edges of $D_i$ and also at least one extra edge for every pair $(C_j,C_{j+1})$. From this observation it immediately follows that
\begin{equation}
 \label{second_bound}
\ell_i \geq w + h-1.
\end{equation}

Combining \eqref{total_bound} and \eqref{second_bound}, we get that, in order to prove \eqref{arc_ineq}, it suffices to establish the following inequality:
$$
 \frac{w+h - \frac{1}{3}}{(w+1)h+(w+2) \frac{r}{k}+1} \geq \frac{k}{n} = \frac{1}{c + \frac{r}{k}} ,
$$
which is equivalent to
\begin{equation}
\label{eq:w-h-c}
\frac{w+1}{w+h-\frac{1}{3}}h + \frac{1}{w+h-\frac{1}{3}} \leq c + \frac{h-\frac{7}{3}}{w+h-\frac{1}{3}}\cdot \frac{r}{k}.
\end{equation}
If $h\leq c-1$, then, using the fact that $h\geq 2$, we get
$$
 \frac{h-\frac{7}{3}}{w+h-\frac{1}{3}}\cdot \frac{r}{k}\geq -\frac{1/3}{w+h-\frac{1}{3}},
$$
and therefore
$$
 \left(\frac{w+1}{w+h-\frac{1}{3}}h + \frac{1}{w+h-\frac{1}{3}}\right)-\left(c + \frac{h-\frac{7}{3}}{w+h-\frac{1}{3}}\cdot \frac{r}{k}\right)\leq
 -\frac{w}{w+h-\frac{1}{3}} - \frac{h-\frac{7}{3}}{w+h-\frac{1}{3}}\cdot \frac{r}{k}\leq 0.
$$
In the last case $h=c$,~\eqref{eq:w-h-c} transforms to
$$
 1 \leq \left(c - \frac{4}{3}\right)c + \left(c-\frac{7}{3}\right) \frac{r}{k},
$$
which also holds since $c \geq 2$ and $\frac{r}{k} < 1$. Thus, the proof in the case $n \geq 2k$ is completed.

\subsection{ \boldmath $n < 2k$}
\paragraph{Construction of $T(n, k)$, see Fig.~\ref{smalln_examples}.} Let $C_1=(v_1v_2\ldots v_kv_1)$ be the boundary of a $k$-gon. We put all internal $s=n-k < k$ vertices on an internal cycle $C_2 = (u_1u_2\ldots u_su_1)$ (if $s=1$ it is a single vertex, if $s=2$ it is an edge). Then, we `evenly' divide the boundary cycle $C_1$ into $s$ groups of consecutive vertices $V_1, V_2, \ldots, V_s$. We do this by choosing separating edges between sets $V_i$ according to Claim~\ref{cl:1}, to ensure that every arc $A\subset C_1$ of length $\ell$ contains at most $\ell\frac{s}{k}+1$ separating edges. Let us denote $v'_i \in V_{i}$ the first clockwise vertex of $V_i$. Next, we connect each vertex $u_i \in C_2$ with all the vertices in the set $V_i$ and additionally with the vertex $v'_{i+1}$ (the addition in indices is modulo $s$).
Finally, we triangulate the inner cycle $C_2$ so that the bounded degree condition is satisfied --- for instance, using a 'comb' path as in Fig.~\ref{comb_triang}.

\renewcommand{\figurename}{Fig.}
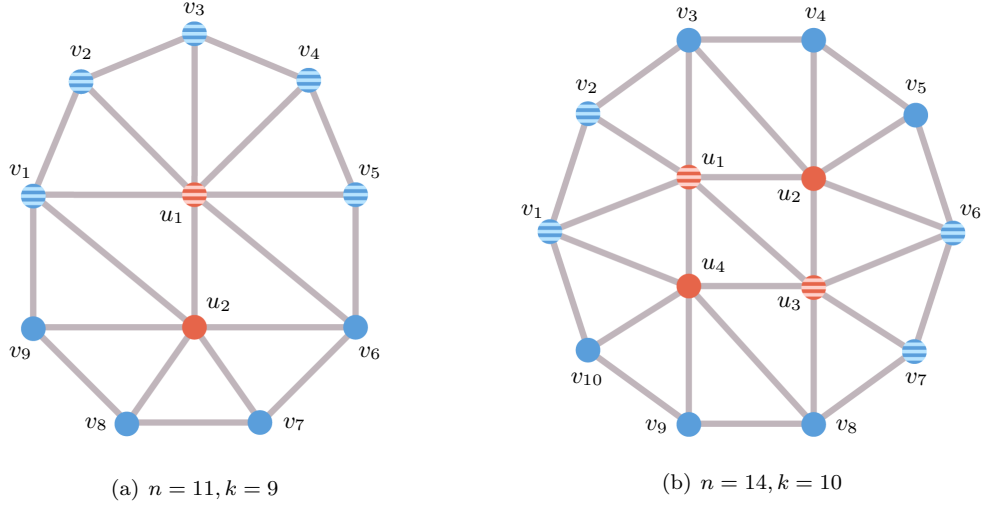
\begin{figure}[htbp]
    \centering
    \subfigure[$n=11, k=9$]
    {
        \begin{minipage}{0.4\textwidth}
            \centering
            \begin{tikzpicture}
                \node[inner sep=0] at (0,0)
                    {\includegraphics[width=\linewidth]{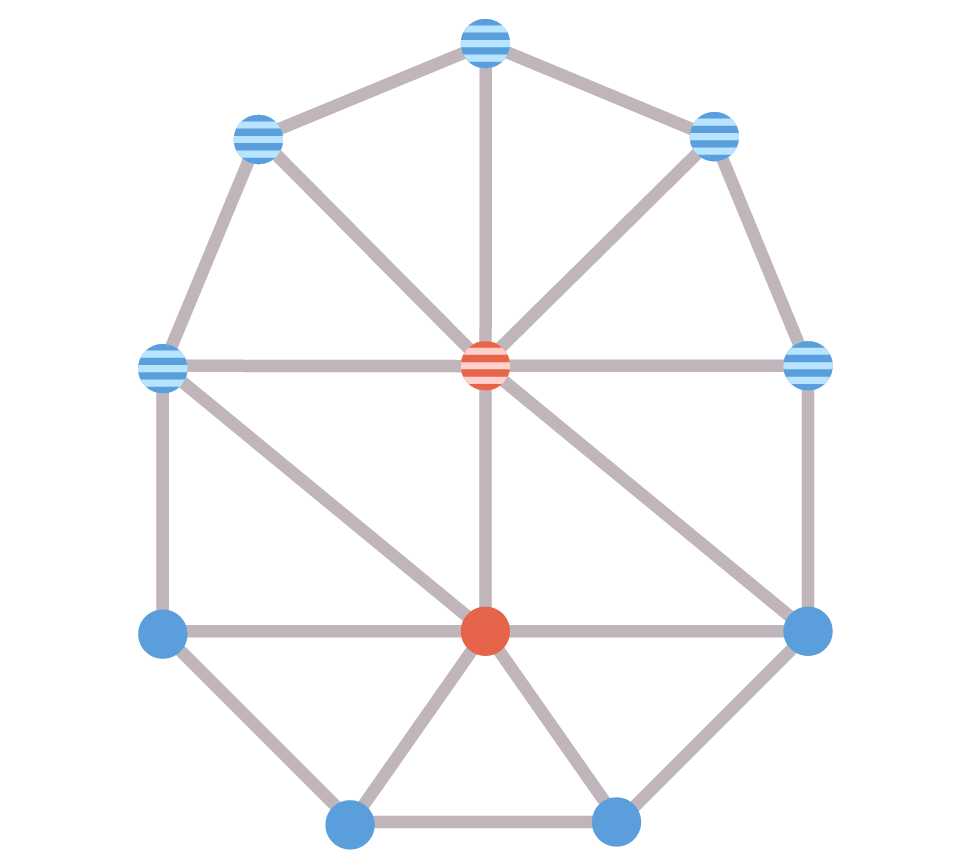}};
                \small
                \node[fill=none, inner sep=2pt] at (1.52,2.32) {$v_{4}$};
                \node[fill=none, inner sep=2pt] at (2.32,0.75) {$v_{5}$};
                \node[fill=none, inner sep=2pt] at (2.32,-1.65) {$v_{6}$};
                \node[fill=none, inner sep=2pt] at (1.32,-2.59) {$v_{7}$};
                \node[fill=none, inner sep=2pt] at (-1.3,-2.59) {$v_{8}$};
                \node[fill=none, inner sep=2pt] at (-2.3,-1.65) {$v_{9}$};
                \node[fill=none, inner sep=2pt] at (-2.3,0.75) {$v_{1}$};
                \node[fill=none, inner sep=2pt] at (-1.5,2.32) {$v_{2}$};
                \node[fill=none, inner sep=2pt] at (0,2.9) {$v_{3}$};
                \node[fill=none, inner sep=2pt] at (-0.3,0.15) {$u_{1}$};
                \node[fill=none, inner sep=2pt] at (0.32,-1.01) {$u_{2}$};
            \end{tikzpicture}
            \vspace{0cm}
        \end{minipage}
    }
    \hspace{0.5cm}
    \subfigure[$n=14, k=10$]
    {
        \begin{minipage}{0.4\textwidth}
            \centering
            \begin{tikzpicture}
                \node[inner sep=0] at (0,0)
                    {\includegraphics[width=\linewidth]{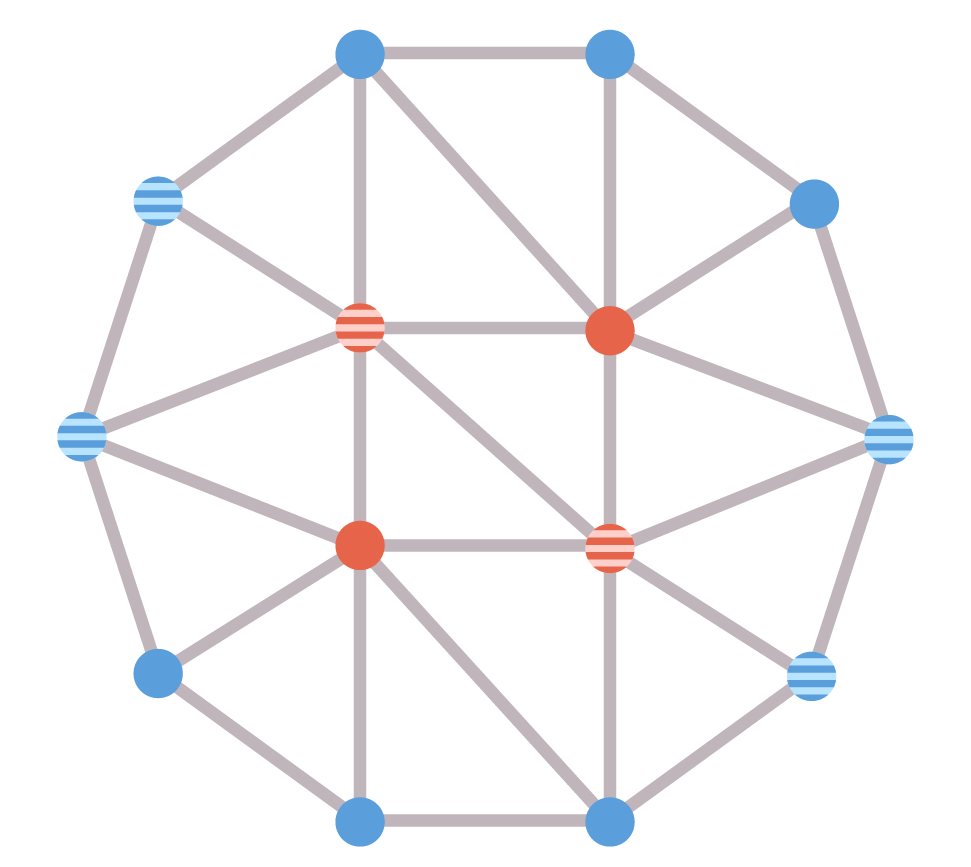}};
                \small
                 \node[fill=none, inner sep=2pt] at (0.85,2.85) {$v_{4}$};
                 \node[fill=none, inner sep=2pt] at (2.18,1.9) {$v_{5}$};
                 \node[fill=none, inner sep=2pt] at (2.9,0.25) {$v_{6}$};
                 \node[fill=none, inner sep=2pt] at (2.18,-1.97) {$v_{7}$};
                 \node[fill=none, inner sep=2pt] at (1.26,-2.6) {$v_{8}$};
                 \node[fill=none, inner sep=2pt] at (-1.26,-2.6) {$v_{9}$};
                 \node[fill=none, inner sep=2pt] at (-2.18,-1.97) {$v_{10}$};
                 \node[fill=none, inner sep=2pt] at (-2.9,0.25) {$v_{1}$};
                 \node[fill=none, inner sep=2pt] at (-2.18,1.9) {$v_{2}$};
                 \node[fill=none, inner sep=2pt] at (-0.85,2.85) {$v_{3}$};
                 \node[fill=none, inner sep=2pt] at (0.5,0.45) {$u_{2}$};
                 \node[fill=none, inner sep=2pt] at (0.5,-0.97) {$u_{3}$};
                 \node[fill=none, inner sep=2pt] at (-0.5,0.92) {$u_{1}$};
                 \node[fill=none, inner sep=2pt] at (-0.5,-0.47) {$u_{4}$};
            \end{tikzpicture}
            \vspace{-0.1cm}
        \end{minipage}
    }
    \caption{Examples of triangulation for $n<2k$: internal vertex and the corresponding $V_i$ are highlighted using the same fill pattern.}
    \label{smalln_examples}
\end{figure}

Note that since $s=\Omega(n)$, we have that $\frac{k}{n} \leq 1 -\varepsilon$ for some fixed $0<\varepsilon<1$. Therefore, the maximum degree of the constructed $(n, k)$-triangulation $T=T(n,k)$ is $\Delta(T) \leq 5+\lceil\frac{k}{n-k}\rceil < 6+\frac{1-\varepsilon}{\varepsilon}$.

\paragraph{Checking condition~\eqref{triang_eq} for $T(n, k)$.}
The proof is based on the observation that subsets of $C_2$ have large neighbourhoods:

\begin{claim}
\label{cl:2case}
 Let $A \subset C_2$ be an arc with $w>0$ vertices. The number of vertices on the external boundary $C_1$ that are adjacent to $A$ is at least $w \frac{k}{s}.$

\end{claim}
\begin{proof}
    The statement trivially holds if $A=C_2$, so we further assume that $w<s$. Let $A=u_iu_{i+1}\ldots u_{i+w-1}$ (the addition in indices is modulo $s$) and let $A'=v'_i\ldots v'_{i+w}$ be the arc consisting of neighbours of $A$ in $C_1$. Without loss of generality, we let $v_2=v'_i$ and $v_{\ell}=v'_{i+w}$, for some $\ell$.  Since $w<s$, the number of neighbors of $A$ on $C_1$ is $\ell-1=\left(\sum\limits_{j=i}^{i+w-1} |V_j|\right) + 1$. We shall prove that this quantity is at least $w\frac{k}{s}$. Assume the opposite: $\sum\limits_{j=i}^{i+w-1} |V_j| < w \frac{k}{s} - 1$. This means that the arc $v_1A'$ has length $\ell-1< w \frac{k}{s}$ and contains $w+1$ separating edges, where the first separating edge is $\{v_1,v_2\}$ and the last separating edge is $\{v_{\ell-1},v_{\ell}\}$ (recall that the vertex $v_{\ell}=v'_{i+w}$ is the first vertex of $V_{i+w}$, which $u_{i+w-1}$ is adjacent to). This contradicts Claim~\ref{cl:1}.
    
\end{proof}

As in the case of $n\geq 2k$, in order to complete the proof of the upper bound when $n<2k$, it is enough to prove the following isoperimetric inequality.
\begin{claim}
\label{cl:2nd_case}
    Suppose that for every simple cycle $C\subseteq T$ of length $\ell \leq k-1$ and with $t>0$ vertices strictly 'inside' $C$ the following inequality is true:
    \begin{equation}
    \label{triang_ineq2}
        \frac{\ell-1}{t} \geq \frac{k}{s} .
    \end{equation}
    Then $T$ satisfies the condition~\eqref{triang_eq} with $q=3-\frac{k}{n}$, $\varepsilon=\frac{2}{5}$, and $\delta=\frac{1}{2}$.
\end{claim}

The proof of this claim is verbatim to the proof of Claim~\ref{cl:first_case}. Nevertheless, for the sake of completeness, we present the proof in Appendix~\ref{appendix_b}.

It remains to check the condition of the claim. Let us fix a simple cycle $C \subseteq T$ of length $\ell \leq k-1$. Obviously, all $t$ vertices lying strictly `inside' $C$ belong to the cycle $C_2$. Since $\ell < k$ we have that $|C \cap C_2| \geq 1$. The $t$ inner vertices split into subsets $U_1, U_2, \ldots, U_r$ that form inclusion-maximal arcs of $C_2$. We have $\sum\limits_{i=1}^r |U_i| = t$. The key observation is that, for every $i\in[r]$, the set of neighbours $N(U_i)\cap C_1$ of $U_i$ on $C_1$ belongs to $C$. From Claim~\ref{cl:2case}, we get that $|N(U_i)\cap C_1| \geq |U_i| \frac{k}{s}$ for every $i \in [r]$. Moreover, from the construction of $T$ and since $s<k$, the sets $N(U_i)\cap C_1$ do not overlap. Thus:
$$
 \frac{\ell-1}{t} \geq \frac{\sum\limits_{i=1}^r |N(U_i)\cap C_1| + |C \cap C_2| - 1}{\sum\limits_{i=1}^r |U_i|} \geq \frac{\sum\limits_{i=1}^r |U_i| \frac{k}{s}}{\sum\limits_{i=1}^r |U_i|} \geq \frac{k}{s}.
$$
Thus, we complete the proof in the case $n < 2k$. The next two sections cover the remaining case where almost all vertices belong to the boundary: $s(n)=o(n)$.

%% file: the_whole_border.tex
\section{Upper bound: \boldmath $s(n) = O(n^{\varepsilon})$}
\label{small_s}
Let $s(n) \leq n^{\varepsilon}$ for some constant $\varepsilon \in (0,1)$. In this section we again derive the upper bound in Theorem~\ref{th:main} by applying Theorem~\ref{main_th} to a fixed triangulation $T=T(n,k) \in \mathcal{T}_{n,k}$ with a bounded maximum degree $\Delta(T)\leq 5$. To do this, we show directly that the hypergraph $\mathcal{H}=\mathcal{H}(T)$ of all copies of $T$ is $(q=C'n^{-1/2}, \beta, \delta)$-superspread with some $C', \beta, \delta > 0$. This implies the desired upper bound $p_c\leq Cq \leq C C' n^{-1/(2+s)}$. 

\paragraph{Construction of $T(n,k)$, see Fig~\ref{big k}.} We start with a triangulation of a $k$-gon by a 'comb' path as in the square of a Hamilton
cycle (see Fig.~\ref{comb_triang}). Then we put the remaining $s$ vertices `evenly' into the $k-2$ triangular faces, splitting each chosen triangle into three. Let us denote by $S=S(T)$ the set of these $s$ internal vertices. We denote the obtained triangulation by $T:=T(n,k)$. Clearly, $\Delta(T)\leq 5$. We also bound the number of internal edges between any pair of vertices from $S$ from below by 
\begin{equation}
\label{border_length}
    B(T) := \left\lfloor\frac{k-2}{s}\right\rfloor,
\end{equation}
as on Fig.~\ref{big k}.

\renewcommand{\figurename}{Fig.}
\begin{figure}[htbp]
    \centering
    \begin{minipage}{0.8\textwidth}
        \centering
        \includegraphics[width=\linewidth]{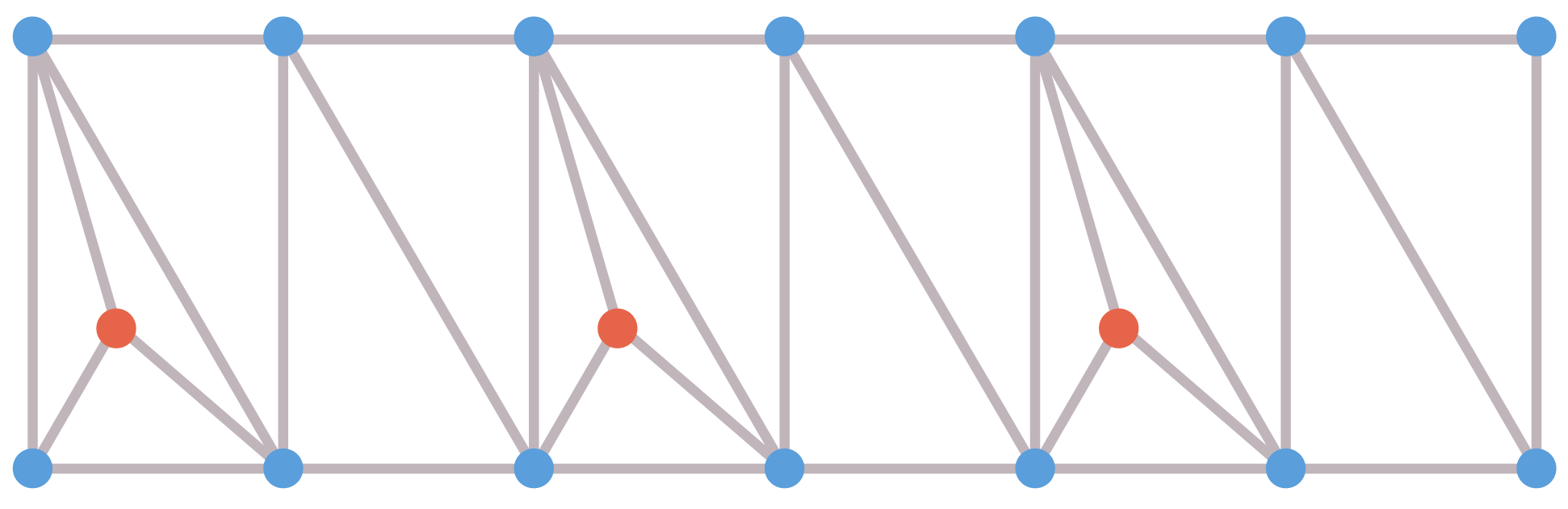} 
    \end{minipage}
    \caption{Example of a triangulation for $n=17$, $k=14$: internal $s=3$ vertices are in red and $B(T)=4$.}
    \label{big k}
\end{figure}

\paragraph{Checking the superspreadness condition.}
Lemma~\ref{lm: q-spread property} is not applicable because of the presence of dense subgraphs --- namely, 4-cliques. Nevertheless, we are able to overcome this issue since these cliques are very far from each other. 

For a subgraph $I \subset T$, denote by $r(I)$ the number of 4-cliques in $I$. The following holds:
\begin{claim}
    For any non-empty subgraph $I\subset T$ with $c$ components

\begin{equation}
\label{general ineq}
    |I| \leq 2v(I) + r(I) - 3c.
\end{equation}
\end{claim}
\begin{proof}
    Let us first consider any connected subgraph $J \subset T$ with $r(J)=0$. It is easy to see that $J$ is 2-degenerate, implying $|J| \leq 2v(J) - 3$.

    Now, let $I_1, I_2, \ldots, I_c$ be all connected components of $I$. We delete from $I$ all $r(I)$ `centers' of $K_4$-patterns in $T$, so for the resulting induced subgraph $I'\subset I$ we have that $|I'|=|I|-3r(I)$ and $v(I')=v(I)-r(I)$. Note that erasing these vertices does not spoil the connectivity of each component $I_i, i\in[c]$, so we get new components by $I'_i, i\in[c]$. For these components we have already verified that $I'_i\leq 2v(I'_i)-3$. Therefore,
    $$|I|-3r(I) = |I'| \leq 2v(I')-3c= 2(v(I)-r(I)) -3c,$$
    which implies the desired inequality.
\end{proof}

We observe also that the big value of $r(I)$ implies the big size of $I$:
\begin{claim}
    For any connected subgraph $I\subset T$ with $r(I) > 0$
    \begin{equation}
    \label{big_r_value}
        |I| \geq (r(I)-1)B(T)/2.
    \end{equation}
\end{claim}
\begin{proof}
     The bound immediately follows from the fact that the edge distance between two `centers' of $4$-cliques is at least $B(T)/2$.
\end{proof}

The superspreadness condition~\eqref{superspread} that implies the required upper bound on $p_c$ is stated as a separate lemma below. Its proof is similar to that of Lemma~\ref{lm: q-spread property}, with minor modifications concerning the number of copies of $K_4$ in a subgraph.

\begin{lemma}
    The hypergraph $\mathcal{H}$ is $(60e^2n^{-\frac{1}{2}}, \min(1/2, 1-\varepsilon), 1)$-superspread for large enough $n$.
\end{lemma}

\begin{proof}
    Fix $I\subseteq T$. Then $|\mathcal{H}\cap\langle I\rangle|$ is the total number of ways to extend $I$ to a copy of $T$ in $K_n$. Or, in other words, it is the number of permutations $\sigma$ (right cosets) in $S_n/\mathrm{Aut}(T)$ such that $I\subset \sigma(T)$. We note that, if $K\subset I$ is a $4$-clique, then $\sigma(K)\cap S(\sigma(T))$ is a single vertex. 

    Suppose that $I$ consists of $c$ connected components $I_1, \ldots, I_c$. Let $0 \leq c_S \leq c$ of them contain a $K_4$ --- without loss of generality, these are $I_1,\ldots,I_{c_S}$. Let us specify a single root $x_i$ in every $I_i$ in such a way that if $K_4\subset I_i$, then $x_i$ belongs to a 4-clique. There are at most $\prod_{i=1}^c v(I_i)\leq \prod_{i=1}^c (|I_i|+1)\leq 2^{|I|}$ ways to specify roots. 

    For $i\in[c_S]$, we force $\sigma(x_i)\in S$. Therefore, the images of $x_i$ can be specified in at most $s^{c_S}$ ways. Next, we reveal the entire $\sigma(I_i)$ for such $i$. Since $I_i$ is connected and $\Delta(T)\leq 5$, we can place each next vertex of $I_i$ in any order that respects connectivity of $I_i$. In other words, we embed an $x_i$-rooted spanning subtree of $I_i$ edge be edge, where each next edge has at most 5 available positions. Therefore, the total number of embeddings is at most $5^{|I_i|}$.
    
    Then, we specify the action of $\sigma$ on $n-v(I)+c-c_S$ remaining blocks $V(I_{c'+1}),\ldots,V(I_c),x\in[n]\setminus V(I)$ in $(n-v(I)+c-c_S)!$ ways. Finally, it remains to reveal the positions of vertices within every block $\sigma(I_i)$, $i\in[c]\setminus[c_S]$. We follow these blocks one by one, according to the specified order dictated by $\sigma$. For each unrevealed $I_i$, we first place its root at the first unoccupied vertex. Since $I_i$ is connected and $\Delta(T)\leq 5$, as above we get that the total number of embeddings is at most $5^{|I_i|}$.
    
    We conclude that
    $$
    |\mathcal{H} \cap \langle I \rangle| \leq \frac{10^{|I|}}{|\mathrm{Aut}(T)|} s^{c_S} (n-v(I)+c-c_S)! .
    $$

    Now we are ready to prove the (super)spreadness property. Substituting~\eqref{general ineq}, we get:
\begin{align*}
    |\mathcal{H} \cap \langle I \rangle| &\leq \frac{10^{|I|}}{|\mathrm{Aut}(T)|} s^{c_S} \left(n-\frac{|I|}{2} + \frac{r(I)}{2}-\frac{1}{2}c -c_S\right)! \\
    &\leq (1+o(1))\frac{10^{|I|}}{|\mathrm{Aut}(T)|} s^{c_S} n! \left(\frac{e}{n}\right)^{\frac{|I|}{2}+\frac{c+c_S}{2}-\frac{r(I)-c_S}{2}} \leq \\ &\leq (1+o(1))\frac{(10e^2)^{|I|}}{|\mathrm{Aut}(T)|} n^{-\frac{|I|}{2}} \left(\frac{s} {n}\right)^{c_S} n^{\frac{c_S-c}{2}} n^{\frac{r(I)-c_S}{2}} n!.
\end{align*}
Summing~\eqref{big_r_value} over all $c_S$ components, we get $|I| \geq (r(I)-c_S)B(T)/2$. Also we note that 
$$
(s/n)^{c_S}n^{(c_S-c)/2} \leq n^{(\varepsilon - 1/2)c_S - c/2} \leq n^{-\min(1/2,1-\varepsilon)c}
$$ 
since $c_S \leq c$ and $s=s(n)\leq n ^{\varepsilon}$. Substituting it, we get 
$$|\mathcal{H} \cap \langle I \rangle| \leq (1+o(1))\frac{(10e^2)^{|I|}}{|\mathrm{Aut}(T)|} n^{-\frac{|I|}{2}} n^{-\min(1/2, 1-\varepsilon)c} n^{\frac{|I|}{B(T)}} n!.$$
Finally, recalling~\eqref{border_length}, we get that $n^{\frac{1}{B(T)}} \leq n^{\frac{s}{k-2-s}} \leq \exp \left\{\frac{n^{\varepsilon} \ln n}{n-2s-2} \right\} < 2$ for large enough $n$. So, since $|\mathcal{H}| = \frac{n!}{|\mathrm{Aut}(T)|}$ and $n \geq \frac{m}{3}$, we get

$$|\mathcal{H} \cap \langle I \rangle| \leq (60e^2 n^{-1/2})^{|I|}  m^{-\min(1/2, 1-\varepsilon)c} |\mathcal{H}| ,$$
as desired.
\end{proof}

\section{Upper bound: \boldmath $n^{1-o(1)} \leq s(n) = o(n)$}
\label{interm_case}

Unlike the previous two cases, the triangulation $T=T(n,k) \in \mathcal{T}_{n, k}$ presented in this section has an unbounded maximum degree. Thus, we cannot directly apply Theorem~\ref{main_th}. Instead, we will apply a more general framework by Spiro~\cite{spiro}. To this end, we find an integer constant $t$ and numbers $m_2,\ldots,m_t$ such that the
 $m$-uniform hypergraph $\mathcal{H}=\mathcal{H}(n,k)$ of all copies of $T$ is $(q; m, m_2, \ldots, m_t, 1)$-spread for $q=O\left(n^{-\frac{1}{3-\alpha}}\right)$.
By Theorem~\ref{spiro-theorem}, 
 this implies that, for some constant $K_0$, $p_c(\mathcal{T}_{n,k}) \leq 2K_0tq=O\left(n^{-\frac{1}{3-\alpha}}\right)$, as desired.

The reason for the bounded degree restriction in the statement of Theorem~\ref{main_th} is that the factor $\Delta^{|I|}$ naturally appears in its proof due to the application of Lemma~\ref{lm:subgraphs_count-rooted}. Let us recall that the main ingredient in this proof --- a fragmentation process --- essentially relies
on two quantities, the number of subgraphs and the number of extensions of a specific subgraph $I$ to a full copy of $T$. Both quantities are estimated via three parameters of $I$: the number of vertices, the number of edges, and the number of connected components. We manage to get rid of the maximum degree condition by introducing another parameter that evaluates a trade-off between the two quantities --- the number of components in the $I$-neighbourhood of a vertex with unbounded degree.

\paragraph{Construction of $T(n,k)$, see Fig~\ref{big k, case2}.}

Similar to the previous section, we insert each internal vertex $v \in S(T)$ into a specific cycle $C(v)$ and connect $v$ with all vertices of $C(v)$. However, in this case the cycles $C(v)$ have non-constant length $\ell=\ell(n)$. We define these lengths in a way such that any two faces bounded by consecutive cycles are separated by exactly $\ell$ vertices. More formally, $T(n, k)$ is a union of blocks $W_1,B_1,W_2,B_2,\ldots,B_{s-1},W_{s}, B'_s$, where
\begin{itemize}
    \item each next block appears to the right of the previous block, they share a single edge, and any two non-consecutive blocks do not have common vertices;
    \item each $W_i$ is a `wheel' graph, i.e., a cycle $C_i$ of length $\ell$ with one vertex $s_i \in S$ in the center connected to all vertices in $V(C_i)$;
    \item each $B_i$ is a `comb' triangulation,  with $\ell+4$ vertices (see Fig.~\ref{comb_triang}) unless $i=s$: the number of vertices in $B'_s$ belongs to $[2,2\ell+1]$.
\end{itemize}

There may be several ways to choose a suitable $\ell$, for which it is possible to construct such a triangulation. But any such choice satisfies the following inequalities:
\begin{equation}
\label{l-s connection}
    n - \ell + 1 \leq s + 2 \ell s \leq n + \ell,
\end{equation}

\renewcommand{\figurename}{Fig.}
\begin{figure}[htbp]
    \centering
    \begin{minipage}{0.8\textwidth}
        \centering
        \includegraphics[width=\linewidth]{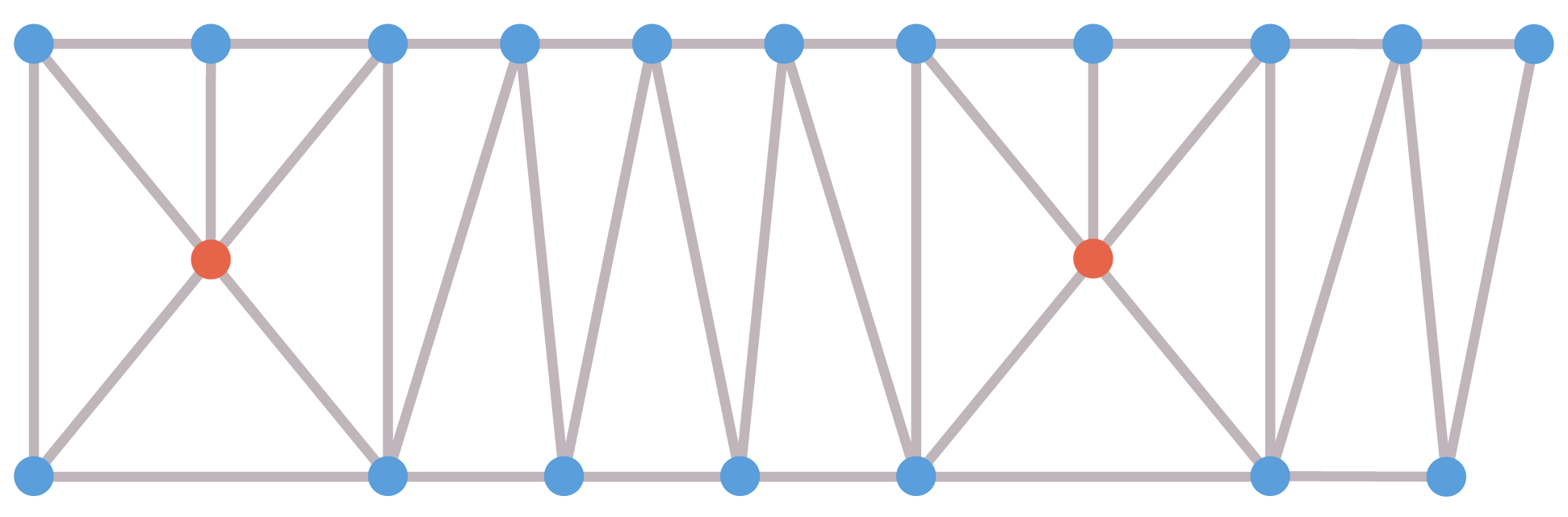} 
    \end{minipage}
    \caption{Example of triangulation for $n=20, k=18, \ell=5$; internal $s=2$ vertices are in red.}
    \label{big k, case2}
\end{figure}

Note that the maximum degree $\Delta(T) = \ell$ is non-constant, but all vertices from $V(T) \backslash S(T)$ have degrees at most $5$. The number of internal edges between any pair of vertices from $S(T)$ is bounded by $B(T) := \ell+3$. Also,  note that $\ell= n^{o(1)}$ since $s\geq n^{1-o(1)}$.

\paragraph{Checking the requirement of Theorem~\ref{spiro-theorem}.} For the rest of the proof, for $I \subseteq T$, we denote by $c(I)$ the number of components in $I$, by $0 \leq c_S(I) \leq c(I)$ the number of components that intersect $S=S(T)$, and by $g(I) \geq c_S(I)$ the number of vertices in $V(I) \cap S$. 

In order to prove the spread condition, we need to bound two values as mentioned above: the number of subgraphs $I \subseteq T$ with specified parameters and the number of ways to extend a fixed subgraph $I$ to a copy of $T$ from $\mathcal{H}$. Both quantities depend on the values of the parameters of $I$, where the new parameter $t(I)$ --- defined below --- plays a crucial role.

For any vertex $v\in V(I) \cap S$, let $t_I(v)\in[0,\ell]$ be the total number of components (disjoint arcs) in $I[N(v)]$ --- the neigbourhood of $v$ in $I$. More precisely,
\begin{equation}
    \nonumber
    t_I(v):=
\begin{cases}
0,\text{if }C(v)\subseteq I; \\
c(I[N(v)]) \text{, otherwise.}
\end{cases}
\end{equation}
\noindent
Set 
$$
 t(I) = \sum\limits_{v \in V(I) \cap S} t_I(v).
$$
Clearly, $t(I) \leq |I|$. As we show in what follows, $t(I)$ can be used naturally to estimate both the number of ways to choose a subgraph in every cycle $C(v)$, $v\in S$, and the number of ways to extend the union of such subgraphs to the union of $C(v)$.

\noindent \paragraph{Estimating the number of subgraphs with certain parameters.} A sufficient estimate for the number of subgraphs is stated below.  Since Definition~\ref{def:spiro_spread} requires checking the spread condition for an arbitrary subset $A \subseteq V(\mathcal{H})$ such that $\langle A \rangle \cap \mathcal{H} \neq \emptyset$, we state our bound for a subgraph $J \subseteq T$.

\begin{lemma}
\label{lm:subgraphs_count}
    Let $J$ be an arbitrary subgraph of $T$ with $|J|=j$. The number of subgraphs $I$ of $J$ with $|I|=i$, $c(I)=c$, and $t(I)=t$ is at most
    $$(2048e)^{i} {2j \choose c} {\ell}^{t}.$$
\end{lemma}

\begin{proof}
    We choose a subgraph $I\subset J$ in the following way:
    \begin{enumerate}
    \item choose a single root from $V(J)$ of every connected component of $I$ so that the root of every component that contains a vertex from $S$ is its leftmost vertex from $S$;

    \item for every connected component $I_h\subset I$ that has a vertex from $S$, specify the number of wheels $w_h$ that this component intersects;

    \item for every connected component $I_h$, choose the set of its internal vertices $V(I_h)\cap S$;

    \item specify the values of $t_I(v)$ for all $v\in S\cap V(I)$;

    \item embed all edges of $I$ but those that touch roots from $S$;

    \item embed the remaining edges of $I$.
    \end{enumerate}

    The first step can be done in at most ${v(J) \choose c} \leq {2j \choose c}$ ways. As soon as the roots are chosen, the number of components that contain vertices of $S$ is identified. Denote it by $c_S$.
    
 Next, we recall that the minimum edge distance between two `wheels' is at least $\frac{B(T)}{2} \geq \frac{\ell}{2}$. Therefore,
  we establish that
 $$\sum\limits_{h=1}^{c_S}w_h \leq c_S(I) + \frac{2|I|}{\ell}.$$ So the number of ways to specify $w_1, w_2, \ldots, w_{c_S}$ is bounded by $2^{c+2i/\ell} \leq 2^{2i}$.
 
    Then, we reveal the positions of the remaining `intermediate' internal vertices in $V(I)\cap S$ in at most $2^{i}$ ways.

    Since $t=\sum_{v\in V(I)\cap S}t_I(v)$ and all $t_I(v)$ are non-negative integers, the number of ways to specify all $t_I(v)$ is at most ${t+g(I)-1\choose g(I)-1}\leq 2^{t+g(I)}\leq 2^{2i}$.

In order to perform the penultimate step, we shift roots from centers of wheels to its disjoint arcs from $I$. This is where parameters $t_v$ play the crucial role. For each internal vertex $v \in I \cap S$ we choose the `leftmost' vertices in arcs (connected components) of $N_I(v)$ as root vertices in at most $\ell^{t_v(I)}$ ways, which gives us an additional factor of $\ell^{t}$. In the case when $C(v) \subseteq I$ we choose the root in $N(v)=C(v)$ arbitrarily. Let $t \leq r \leq g(I)+t$ be the total number of chosen roots. Now, it suffices to embed at most $c-c_S+r$ rooted components with fixed roots (some component may contain several roots) and with the total number of edges at most $i-t$. To this end, we first fix the sizes of these components $i_h$, $1 \leq h \leq c-c_S+r$, in at most ${i-t+c-c_S + r \choose c-c_S+r} \leq 2^{i+c+g(I)} \leq 2^{3i}$ ways. Then, for every $h\leq c-c_S+r$, we expose the $h$-th component. From Lemma~\ref{lm:subgraphs_count-rooted}, we get that it can be done in at most $(4e)^{i-t}$ ways.

 Finally, it remains to attach centers of wheels to embeded arcs --- or, in other words, reveal $N_I(v)$, for $v\in S$.  Note that, for each such arc, we have already identified its leftmost vertex. So, we only need to specify the length of each arc. The set of arcs in $\cup_c (C(v)\cap I)$ is known and has total length at most $i$. Therefore, the number of ways to sepcify these lengths is at most $2^i$, completing the construction of $I$. 
 
Summing up, we get that the number of subgraphs $I$ with given parameters is at most
    $$ {2j \choose c} 2^{9i} (4e)^{i-t} \ell^t \leq (2048 e)^i {2j \choose c} {\ell}^t,
    $$
as desired.
\end{proof}

\vspace{\baselineskip}

\noindent \paragraph{Estimating the number of ways to extend a subgraph.}
First, we evaluate densities of subgraphs of $T$ in terms of the introduced parameters:
\begin{claim}
    \label{subgraphs relations, dense case}
    Let $I$ be a subgraph of $T$ with $|I|=i$, $v(I)=v$, $c(I)=c$, $c_S(I) = c_S$, and $t(I)=t$. Then $$v \geq \frac{i}{2} + c + \frac{c-c_S}{2} + \frac{t}{2}.$$
\end{claim}
\begin{proof}
    Let $I_j$ be a connected component of $I$ that contains internal vertices from $S$. For every vertex $u\in S\cap I_j$  the induced subgraph $I_j(u):=I_j \cap (N_I(u)\cup u)$ satisfies
    \begin{equation}
    \label{wheel_subgraphs_density}
        |I_j(u)| \leq 2v(I_j(u)) - 2 - t_I(u).
    \end{equation}
    Indeed, for the whole 'wheel'-subgraph $W_u$ around $u$ we have $|W_u| = 2v(W_u) - 2$ but the presence of $t_I(u)$ disconnected components on $C(u)$ decreases the number of edges by $t_I(u)$. Summing up the inequalities in~\eqref{wheel_subgraphs_density} over all vertices in $S \cap I_j$ and observing that completing the remaining subgraph of $I_j$ between two consecutive wheels on, say, $x$ vertices costs at most $2x+2$ edges, we get that
    \begin{equation}
    \label{subg_rel_eq1}
        |I_j| \leq 2v(I_j) - 2 - t(I_j),
    \end{equation}
    Next, the remaining $c-c_S$ components $I_j$ that do not contain vertices from $S$ are 2-degenerate, thus 
    \begin{equation} 
    \label{subg_rel_eq2}
        |I_j| \leq 2v(I_j) - 3.
    \end{equation}
    Summing (\ref{subg_rel_eq1}, \ref{subg_rel_eq2}) over all $c$ components of $I$ we get the statement.
\end{proof}

Now we provide a sufficient bound for the number of extensions.
\begin{lemma}
\label{lm:extensions-count-third-case}
     Let $I$ be a subgraph of $T$ with $|I|=i$, $v(I)=v$, $c(I)=c$, $c_S(I) = c_S,$ and $t(I)=t$. Then
     $$|\mathcal{H} \cap \langle I \rangle| \leq 10^{i} \ell^{t} \left(n-\frac{i}{2}-\frac{t}{2} - \frac{c-c_S}{2}\right)! .$$
\end{lemma}
\begin{proof}
    First, we fix roots in $c$ components in at most $2^i$ ways. Then, we order the set of connected components and remaining isolated vertices in $(n-v+c)!$ ways. It remains to count the number of ways to embed each component of $I$ into $T$ when the leftmost vertex (the root) has been already embedded. We embed each connected component edge by edge, in an arbitrary order that preserves connectivity, unless the leftmost vertex $x$ of current non-embedded edge belongs to $S$. In the latter case, we distinguish between the case when $t_I(x)=0$ and $t_I(x)>0$. When $t_I(x)=0$ and $x$ is the root, there are at most $\ell<2^{\ell}$ ways to embed the entire wheel. When $t_I(x)=0$ and $x$ is not the root, we have already embeded some edge that touches $x$, so the number of ways to embed the remaining part of the wheel is at most $2<2^{\ell}$. Finally, when $t_I(x)>0$, we only need to embed each leftmost vertex of each arc, which can be done in at most $\ell^{t_I(x)}$ ways. We also observe that when $x\notin S$, its degree in $T$ is at most 5, thus every edge growing from $x$ can be embedded in at most 5 ways. Summing up, we get
\begin{equation}
\nonumber
\label{ext_number}
    10^{i} \ell^{t} (n-v+c)!
\end{equation}
ways. From Claim~\ref{subgraphs relations, dense case} we finally get
\begin{equation}
\label{ext result}
    |\mathcal{H} \cap \langle I \rangle| \leq 10^{i} \ell^{t} \left(n-\frac{i}{2}-\frac{t}{2} - \frac{c-c_S}{2}\right)! .
\end{equation}
\end{proof}

\paragraph{Completing the proof.}

Recall Definition~\ref{def:spiro_spread} and~\eqref{eq:M_j-A}. The following lemma, combined with Theorem~\ref{spiro-theorem}, completes the proof of the upper bound in Theorem~\ref{th:main} in the case $n^{1-o(1)}\leq s(n)=o(n)$.

\begin{lemma} There exists a constant $C>0$ such that
    $\mathcal{H}(n, k)$ is $(Cn^{-\frac{1}{3-\alpha}}; m, \frac{m}{n^{1/5}}, \frac{m}{n^{2/5}}, \ldots, \frac{m}{n}, 1)$-spread.
\end{lemma}
\begin{proof}
    Fix an arbitrary subgraph $J\subset T$ with $j$ edges. Then Lemma~\ref{lm:subgraphs_count} and Lemma~\ref{lm:extensions-count-third-case} imply
    $$
    M_i (J)  \leq \sum_{c=1}^i \sum_{c_S=0}^c \sum_{t=0}^i (20480e)^i {2j \choose c}\ell^{2t} \left(n-\frac{i}{2} - \frac{t}{2} - \frac{c-c_S}{2}\right)! .
    $$
Since $T$ has a constant number of automorphisms $f$, we get that $|\mathcal{H}| = \frac{n!}{f}$. Therefore
\begin{equation}
\label{f_i specific}
    M_i (J) \leq e^{O(i)} \sum_{c=1}^i \sum_{c_S=0}^c \sum_{t=0}^i \left(\frac{j}{c\sqrt{n}}\right)^{c-c_S} \left(\frac{j}{c}\right)^{c_S} \left(\frac{\ell^2}{\sqrt{n}}\right)^t n^{-\frac{i}{2}} |\mathcal{H}|.
\end{equation}
Now we split the set of components of $I$ which contain vertices from $S$ into two groups: $0 \leq c_t \leq c_S$ of them have $t(I_j) > 0$ and all the other components have $t(I_j) = 0$. Since every component in the second group contains a complete wheel, we get $c_S-c_t \leq \frac{i}{2\ell}$. Moreover,~\eqref{l-s connection} implies $\frac{1}{5\ell} \leq \frac{s}{2n+s}$ for all $n$ large enough. Therefore,
$$
\left(n^\frac{1}{5}\right)^{c_S-c_t} \leq n^{\frac{i}{10\ell}}\leq n^{\frac{si}{2(2n+s)}}.
$$
We have that $\frac{1}{2}-\frac{1}{3-\alpha}=\frac{1-\alpha}{2(3-\alpha)}=\frac{s}{2(2n+s)}$.  Thus,
\begin{align*}
         \frac{M_i (J)}{|\mathcal{H}|} &\leq e^{O(i)} \sum_{c=1}^i \sum_{c_S=0}^c \sum_{c_t=0}^{c_S} \left(\frac{j}{c\sqrt{n}}\right)^{c-c_S} \left(\frac{j \ell^2}{c\sqrt{n}}\right)^{c_t} \left(\frac{j}{cn^{1/5}}\right)^{c_S-c_t}  n^{\frac{si}{2(2n+s)}} n^{-\frac{i}{2}}\\ & \leq  e^{O(i)} \sum_{c=1}^i \left(\frac{j}{c n^{1/5}}\right)^c n^{-\frac{i}{3-\alpha}},
\end{align*}
Now note that from $\left(\frac{x}{c}\right)^c \leq e^{\frac{x}{e}}$, it follows that if $\frac{j}{n^{1/5}} \leq i$ then $M_i(J) \leq \left(Cn^{-\frac{1}{3-\alpha}}\right)^i |\mathcal{H}|$ for some constant $C$. By definition, it means that $\mathcal{H}$ is $(Cn^{-\frac{1}{3-\alpha}}; m, \frac{m }{n^{1/5}}, \frac{m }{n^{2/5}}, \ldots, \frac{m }{n}, 1)$-spread, completing the proof of the lemma.
\end{proof}

%% file: conclusion.tex
\section{Further questions}
\label{sc:further}

It would be interesting to determine the sharp threshold, that is, the exact asymptotics of $p_c(\mathcal{T}_{n, k})$. Our current approach appears insufficient for this purpose, since the family $\mathcal{T}_{n,k}$ contains exponentially many non-isomorphic triangulations. Even in cases where a family consists of a single isomorphism class --- as for the square of a Hamilton cycle --- determining the exact constant factor is already a highly nontrivial problem. Another natural direction for future research is to determine threshold probabilities for triangulations of other surfaces, as well as for other spanning subgraphs embedded on surfaces.

A related problem, concerning the threshold for the appearance of a spanning triangulation of a fixed $n$-gon was studied in~\cite{Brett}. This question pertains to dense random graphs with constant $p$ and requires completely different techniques. It would be interesting to investigate the analogous question for boundaries of arbitrary lengths.

\section*{Acknowledgements}

The first author expresses gratitude to Fedor Petrov for the opportunity to present an analysis of articles on the topic of this paper at the seminar on Algebraic Combinatorics, and Annika Heckel for correcting inaccuracies in an earlier version of this paper and for helpful discussions. Also, the first author thanks Elizaveta Tsiplakova for helping draw pictures.

%% file: appendix.tex
\appendix
\section{Proof of Lemma~\ref{lm: q-spread property}}
\label{appendix}

First, note that the hypergraph $\mathcal{H}$ of all copies of $F$ is $(k:=|E(F)|)$-uniform and that $|\mathcal{H}| = \frac{n!}{|\mathrm{Aut}(F)|}$. Now, fix arbitrary subgraph $I$ of $K_n$. Clearly, we can consider only the case where $I$ is contained in some copy of $F$. Note that every copy of $F$ in $T$ is completely determined by the ordering of vertices, and exactly $|\mathrm{Aut(F)}|$ such permutations define the same copy of $F$.  So it is enough just to estimate the number of permutations that define a copy of $F$ containing $I$.

Suppose that $I$ consists of $c$ connected components $I_1, \ldots, I_c$. Let us fix root vertices for each component: $x_1, \ldots, x_c$. The total number of ways to do this is
\begin{equation}
\label{number_comps}
    \prod\limits_{i=1}^c v(I_i) \leq \prod\limits_{i=1}^c (|I_i| + 1) \leq  \prod\limits_{i=1}^c 2^{|I_i|} = 2^{|I|} .
\end{equation}

Each permutation induces an ordering $\pi$ of the chosen roots and $n-v(I)$ vertices not included in the fragment $I$. The total number of such induced permutations is $(c+n-v(I))!$. For each fixed $\pi$, we estimate the total number of ways to 'insert' all remaining vertices in a way which is consistent with the structure of components $I_j$. To do this, we consider labelling of vertices in each component $I_j$ starting from the root $x_j$ such that each vertex has at least one neighbour in its component with a smaller label. Then we insert vertices into the permutation following this labelling. Note that at each step there are at most $d$ choices, as $\Delta(I_j) \leq d$. So the total number of possible permutations with each fixed $\pi$ is at most
\begin{equation}
\label{comp_order}
    \prod_{i=1}^c d^{v(I_j) -1} \leq d^{|I|} .
\end{equation}
From (\ref{number_comps}),(\ref{comp_order}) we get that:
$$|\mathcal{H} \cap \langle I \rangle| \leq \frac{(2d)^{|I|}}{|\mathrm{Aut(F)}|} (n-v(I)+c)! .$$

Let us check the spreadness property. Summing the inequality (\ref{triang_eq}) over all $c$ components, we get that $v(I) \geq \frac{|I|}{q} + c$. Substituting it:
\begin{equation}
 \begin{gathered}
 \nonumber
|\mathcal{H} \cap \langle I \rangle| \leq \frac{(2d)^{|I|}}{|\mathrm{Aut(F)}|} \left(n-\frac{|I|}{q}\right)! \leq \frac{2(2d)^{|I|}}{|\mathrm{Aut(F)}|} e^{\frac{|I|}{q}} n^{- \frac{|I|}{q}} n! \leq \left(4d C_2^{\varepsilon}e^{\frac{1}{q}} n^{\frac{-1}{q}}\right)^{|I|} |\mathcal{H}|,
\end{gathered}
\end{equation}
which proves that $\mathcal{H}$ is $8d C_2^{\varepsilon}e^{1/q} n^{-1/q}$-spread.

Now it remains to check the superspreadness condition (\ref{superspread}). We suppose that $|I| \leq \delta n \leq \frac{\delta}{C_1} k$. Again, from~\eqref{triang_eq} we get that $v(I) \geq \frac{|I|}{q} + (1 + \varepsilon) \cdot c$. So
 \begin{equation}
 \begin{gathered}
 \nonumber
|\mathcal{H} \cap \langle I \rangle| \leq \frac{(2d)^{|I|}}{|\mathrm{Aut(F)}|} \left(n-\frac{|I|}{q}-c\varepsilon\right)! \leq \frac{4(2d)^{|I|}}{|\mathrm{Aut(F)}|} e^{c\varepsilon + \frac{|I|}{q}} n^{-c\varepsilon - \frac{|I|}{q}} n!.
\end{gathered}
\end{equation}
Using the condition that $ n \geq \frac{k}{C_2}$ and the trivial observation $c \leq |I|$ we get
$$
|\mathcal{H} \cap \langle I \rangle| \leq \left(8d C_2^{\varepsilon}e^{\varepsilon+\frac{1}{q}} n^{-\frac{1}{q}}\right)^{|I|} k^{-c\varepsilon} |\mathcal{H}|,
$$
which completes the proof.

\section{Proof of Claim~\ref{cl:2nd_case}}
\label{appendix_b}

We show that $T$ satisfies the condition~\eqref{triang_eq} with $q=3-\frac{k}{n}$, $\varepsilon=\frac{s+k}{2k+3s} > \frac{2}{5}$ since $s<k$, and $\delta=\frac{1}{2}$. Similar to the proof of Claim~\ref{cl:first_case}, it is enough to prove the condition~\eqref{triang_eq} for fragments with a single simple outer face $C \subset T$. Let $\ell$ be the length of the cycle $C$ and $v=v(I)$ be the total number of vertices in $I$. The number of edges in $I$ is at most the number of edges in a triangulation of the $\ell$-gon $C$ with the total number of vertices $v$: $|I| \leq 3v-3-\ell$. Therefore, the inequality
    \begin{equation}
    \label{eq:claim_iso2_v-ell}
    v \geq \frac{3v-3-\ell}{3-\frac{k}{n}} + 1 + \frac{s+k}{2k+3s},
    \end{equation}
    implies the first inequality in (\ref{triang_eq}) with the specified parameters. Since $n=k+s, v=\ell+t$, the latter inequality is equivalent to
    $$\frac{k}{s} \leq \frac{\ell - 1}{t}$$
    if $t > 0$; and it holds if $t=0$.
    Note that the latter inequality holds when $\ell \leq k-1$ due to the assumption \eqref{triang_ineq2}.  Therefore, we have proved~\eqref{triang_eq} except for the case when $v = \ell+t > k \geq \frac{1}{2}n$ and $\ell \geq k$. Assume this last case. Since $I$ is connected and has a cycle, $|I|\geq v\geq \frac{1}{2}n$, as in the condition in the second case of~\eqref{triang_eq}. Since $\ell \geq k$, in the same way as in~\eqref{eq:claim_iso2_v-ell}, it suffices to prove 
    $$v \geq \frac{3v-3-k}{3-\frac{k}{n}} + 1$$
    which is immediate since $v < n + 1$.